\documentclass{article}
\usepackage{amsmath, amssymb, amsthm, tocloft, enumerate, tikz, algorithmicx,
            algpseudocode, mathrsfs, color, comment, xspace, cite}
\usepackage[top=50pt,bottom=60pt,left=78pt,right=76pt]{geometry}
\usepackage[ruled, vlined, linesnumbered]{algorithm2e}
\usepackage{tikz-cd}
\usepackage[colorlinks]{hyperref}
\hypersetup{linkcolor=blue, urlcolor=blue, citecolor=red}

\algrenewcommand{\algorithmiccomment}[1]{\hfill[{\it #1}]}
\newcommand{\N}{\mathbb{N}}

\renewcommand{\ker}{\operatorname{ker}}

\newcommand{\J}{\mathscr{J}}

\newcommand{\set}[2]{\{#1:#2\}}
\newcommand{\dom}{\operatorname{dom}}
\newcommand{\im}{\operatorname{im}}
\newcommand{\genset}[1]{\langle#1\rangle}

\renewcommand{\to}{\longrightarrow}

\newcommand{\GAP}{{\sc GAP}\xspace}

\newcommand{\libsemigroups}{{\sc libsemigroups}\xspace}

\numberwithin{equation}{section}

\newtheorem{prop}[equation]{Proposition}
\newtheorem{lem}[equation]{Lemma}
\newtheorem{cor}[equation]{Corollary}
\newtheorem{defn}[equation]{Definition}
\newtheorem{exam}[equation]{Example}
\newtheorem*{assumption}{Assumptions}
\newenvironment{de}{\begin{defn}\rm}{\end{defn}}

\begin{document}
\title{Two variants of the Froiduire-Pin Algorithm for finite semigroups}
\author{J. Jonu\v{s}as, J. D. Mitchell, and M. Pfeiffer}
\maketitle

\begin{abstract}
  In this paper, we present two algorithms based on the Froidure-Pin Algorithm
  for computing the structure of a finite semigroup from a generating set.  
  As was the case with the original algorithm of Froidure and Pin, the
  algorithms presented here produce the left and right Cayley graphs, a
  confluent terminating rewriting system, and a reduced word of the rewriting
  system for every element of the semigroup.
  
  If $U$ is any semigroup, and $A$ is a subset of $U$, then we denote by
  $\langle A\rangle$ the least subsemigroup of $U$ containing $A$. 
  If $B$ is any other subset of $U$, then, roughly speaking, the first
  algorithm we present describes how to use any information about $\langle
  A\rangle$, that has been found using the Froidure-Pin Algorithm, to compute
  the semigroup $\langle A\cup B\rangle$.  More precisely, we describe the data
  structure for a finite semigroup $S$ given by Froidure and Pin, and how to
  obtain such a data structure for $\langle A\cup B\rangle$ from that for
  $\langle A\rangle$.  The second algorithm is a lock-free concurrent version
  of the Froidure-Pin Algorithm. 
\end{abstract}


\section{Introduction}

A \textit{semigroup} is just a set $U$ together with an associative binary
operation. If $A$ is a subset of a semigroup $U$, then we denote by
$\genset{A}$ the smallest subsemigroup of $U$ containing $A$, and refer to $A$
as the \textit{generators} of $\genset{A}$. If $A$ and $B$ are subsets of $U$,
then we write $\genset{A, B}$ rather than $\genset{A \cup B}$. The question of
determining the structure of the semigroup $\genset{A}$ given the set of
generators $A$ has a relatively long history; see the introductions
of~\cite{East2015ab} or~\cite{Froidure1997aa} for more details.

In~\cite{Froidure1997aa} the authors present an algorithm for computing a
finite semigroup; we refer to this as the \textit{Froidure-Pin Algorithm}. More
precisely, given a set $A$ of generators belonging to a larger semigroup $U$,
the Froidure-Pin Algorithm simultaneously enumerates the elements, produces the
left and right Cayley graphs, a confluent terminating rewriting system, and a
reduced word of the rewriting system for every element of $\genset{A}$. The
Froidure-Pin Algorithm is perhaps the first algorithm for computing an
arbitrary finite semigroup and is still one of the most powerful, at least for
certain types of semigroup.  Earlier algorithms, such as those
in~\cite{Konieczny1994aa, Lallement1990aa}, often only applied to specific
semigroups, such as those of transformations or Boolean matrices. 
In contrast, the Froidure-Pin Algorithm can be applied to any semigroup where
it is possible to multiply and test equality of the elements. 

Green's relations are one of the most fundamental aspects of the structure of a
semigroup, from both a theoretical and a practical perspective;
see~\cite{Howie1995aa} or~\cite{Rhodes2009aa} for more details. The Green's
structure of a semigroup underlies almost every other structural feature. As
such determining the Green's relations of a finite semigroup is a necessary
first step in most further algorithms. Computing Green's relations is
equivalent to determining strongly connected components in the left and right
Cayley graphs of a semigroup, for which there are several well-known
algorithms, such as Tarjan's or Gabow's. Hence the problem of
determining the Green's relations of a semigroup represented by a generating
set can be reduced to the problem of finding the Cayley graphs. Thus the
performance of algorithms for finding Cayley graphs has a critical
influence on the majority of further algorithms for finite semigroups. 

The Froidure-Pin Algorithm involves determining all of the elements of the
semigroup $\genset{A}$ and storing them in the memory of the computer.  In
certain circumstances, it is possible to fully determine the structure of
$\genset{A}$ without enumerating and storing all of its elements. One such
example is the Schreier-Sims Algorithm for permutation groups;
see~\cite{B.-Eick2004aa,Seress2003ab,Sims1970aa}.  In~\cite{East2015ab}, based
on~\cite{Konieczny1994aa, Lallement1990aa, Linton1998aa}, the Schreier-Sims
Algorithm is utilised to compute any subsemigroup $\genset{A}$ of a regular
semigroup $U$. Of course, this method is most efficient when trying to compute
a semigroup containing relatively large, in some sense, subgroups.  In other
cases, it is not possible to avoid enumerating and storing all of the elements
of $\genset{A}$. For example, if a semigroup $S$ is $\J$-trivial, then the
algorithms from~\cite{East2015ab} enumerate all of the elements of $S$ by
multiplying all the elements by all the generators, with the additional
overheads that the approach in~\cite{East2015ab} entails. For semigroups of
this type, the algorithms in~\cite{East2015ab} perform significantly worse than
the Froidure-Pin Algorithm.

In this paper, we present two algorithms based on the Froidure-Pin Algorithm
from~\cite{Froidure1997aa}.  The first algorithm
(Algorithm~\ref{algorithm-closure} (\textsf{Closure})) can be used to extend
the output of the Froidure-Pin Algorithm for a given semigroup $\genset{A}$, to
compute a supersemigroup $\genset{A, B}$ without recomputing $\genset{A}$.
This algorithm might be useful for several purposes, such as for example: in
combination with Tietze transformations to change generators or relations in a
presentation for $\genset{A}$; finding small or irredundant generating sets for
$\genset{A}$; computing the maximal subsemigroups of certain 
semigroups \cite{Donoven2016aa}.

The second algorithm (Algorithm~\ref{algorithm-parallel}
(\textsf{ConcurrentFroidurePin})) is a lock-free concurrent version of the
Froidure-Pin Algorithm.  Since computer processors are no longer getting
faster, only more numerous, the latter provides a means for fully utilising
contemporary machines for computing finite semigroups. 

If $S$ is a semigroup generated by a set $A$, then the time and space
complexity of the Froduire-Pin Algorithm is $O(|S||A|)$.  Hence for example, in
the case of transformation semigroups of degree $n$, the worst case complexity
is at least $O(n ^ n)$. The algorithms presented here do not improve on the
underlying complexity of the Froduire-Pin Algorithm but offer practical
improvements on the runtime.

Both algorithms are implemented in the C++ library
\libsemigroups~\cite{Mitchell2016ab}, which can be used in both \GAP
\cite{Mitchell2016aa} and Python \cite{Elliot2017aa}.

The paper is organised as follows. Some relevant background material relating
to semigroups is given in Section~\ref{section-prelim}.  In
Section~\ref{section-fropin} we describe the Froidure-Pin Algorithm and prove
that it is valid. While there is much overlap between
Section~\ref{section-fropin} and \cite{Froidure1997aa}, this section is
necessary to prove the validity of Algorithms~\ref{algorithm-closure}
and~\ref{algorithm-parallel}, and because some details are omitted from
\cite{Froidure1997aa}. Additionally, our approach is somewhat different to that
of Froidure and Pin's in~\cite{Froidure1997aa}.  The first of our algorithms,
for computing $\genset{A, B}$ given $\genset{A}$, is described in
Section~\ref{section-closure}, and the lock-free concurrent version of the
Froidure-Pin algorithm is described in Section~\ref{section-parallel}.  
Sections~\ref{section-closure} and~\ref{section-parallel} both contain 
empirical information about the performance of the implementation our
algorithms in \libsemigroups~\cite{Mitchell2016ab}, and our implementation of
Froidure and Pin's original algorithm in \libsemigroups. The
performance of \libsemigroups is roughly the same as that of Pin's
implementation in Semigroupe 2.01~\cite{Pin2009ab}.

\section{Preliminaries}\label{section-prelim}

In this section, we recall some standard notions from the theory of semigroups;
for further details see~\cite{Howie1995aa}.

If $f: X \to Y$ is a function, then $\dom(f) = X$ and $\im(f) = f(X) =
\set{f(x)}{x \in X}$.
If $S$ is a semigroup and $T$ is a subset of $S$, then $T$ is a
\textit{subsemigroup} if $ab \in T$ for all $a, b\in T$.
A semigroup $S$ is a \textit{monoid} if it has an identity element, i.e.\ an
element $1_S\in S$ such that $1_Ss=s1_S=s$ for all $s\in S$.  

If $A$ is a subset of a semigroup $S$, then we denote by
$\genset{A}$ the smallest subsemigroup of $S$ containing $A$, and refer to $A$
as the \textit{generators} or \textit{generating set} of $\genset{A}$. If $A$
is a generating set for a semigroup $S$, then the \textit{right Cayley graph}
of a semigroup $S$ with respect to $A$ is the digraph with vertex set $S$ and
an edge from $s$ to $sa$ for all $s\in S$ and $a\in A$. The \textit{left Cayley
graph} of $S$ is defined analogously. 

A \textit{congruence} $\rho$ on a semigroup $S$ is an equivalence relation on
$S$ which is invariant under the multiplication of $S$. More precisely, if $(x,
y) \in \rho$,  then $(xs, ys), (sx, sy)\in \rho$ for all $s\in S$. A
\textit{homomorphism} from a semigroup $S$ to a semigroup $T$ is just a
function $f: S\to T$ such that $f(xy) = f(x)\ f(y)$ for all $x,y\in S$. 
If $\rho$ is a congruence on a semigroup $S$, then the \textit{quotient}
$S/\rho$ of $S$ by $\rho$ is the semigroup consisting of the equivalence
classes $\set{s/\rho}{s\in S}$ of $\rho$ with the operation
$$x/\rho\ y/\rho = (xy)/\rho$$
for all $x,y\in S$.  If $f: S\to T$ is a homomorphism of semigroups, then 
the \textit{kernel} of $f$ is
$$\ker(f) = \set{(x,y)\in S\times S}{f(x) = f(y)},$$
and this is a congruence. There is a natural isomorphism between
$S/\ker(f)$ and $\im(f)$ which is a subsemigroup of $T$. 

An \textit{alphabet} is just a finite set $A$ whose elements we refer to as
\textit{letters}. A \textit{word} is just a finite sequence $w = (a_1, \ldots,
a_n)$ of letters $a_1, \ldots, a_n$ in an alphabet $A$. For convenience, we
will write $a_1\cdots a_n$ instead of $(a_1, \ldots, a_n)$.  The
\textit{length} of a word $w = a_1 \cdots a_n \in A^*$ is $n$ and is denoted
$|w|$.  A non-empty \textit{subword} of a word $a_1\cdots a_n$ is a word of the
form $a_ia_{i + 1}\cdots a_j$ where $1\leq i \leq j \leq n$.

The \textit{free semigroup} over the alphabet $A$ is the set of all words over
$A$ with operation the concatenation of words; we denote the free semigroup on
$A$ by $A^+$.   The free semigroup $A ^ +$ has the property that every function
$f: A \to S$, where $S$ is a semigroup, can be uniquely extended to a
homomorphism $\nu: A ^ + \to S$ defined by $\nu(a_1\cdots a_n) = f(a_1)\
f(a_2)\cdots f(a_n)$ for all $a_1\cdots a_n \in A ^ +$.  We denote the monoid
obtained from $A ^ +$ by adjoining the empty word by $A ^ *$; which is called
the \textit{free monoid}.

Suppose that $<$ is a well-ordering of $A$. Then $<$ induces a well-ordering on
$A ^ *$, also denoted by $<$, where  $u < v$ if $|u| < |v|$ or there exist  $p,
u', v' \in A ^ *$ and $a, b\in A$, $a < b$, such that $u = pau'$ and $v =
pbv'$.  We refer to $<$ as the \textit{short-lex} order on $A ^ *$.  Let $S$ be
a semigroup and let $A\subseteq S$. Then there exists a unique homomorphism
$\nu: A ^+ \to S$ that extends the inclusion function from $A$ into $S$. The
word $u\in A ^ +$ such that $$u = \min\set{v\in A^+}{\nu(u)=\nu(v)},$$ with
respect to the short-lex order on $A ^ +$, is referred to as a \textit{reduced
word} for $S$.  For every word $w \in A ^+$, there is a unique reduced word for
$S$, which we denote by $\overline{w}$. 

We require the following results relating to words and the short-lex order $<$. 

\begin{prop}[cf. Proposition 2.1 in \cite{Froidure1997aa}]
  \label{prop-fropin-1}
  Let $A$ be any alphabet, let $u, v, x, y \in A ^*$,  and let $a, b\in A$.
  Then the following hold:
  \begin{enumerate}[\rm (a)]
    \item 
      if $u < v$, then $au < av$ and $ua < va$;
    \item 
      if $ua \leq vb$, then $u \leq v$; 
    \item 
      if $u \leq v$, then $xuy \leq xvy$. 
  \end{enumerate}
\end{prop}

\begin{prop}[cf. Proposition 2.1 in \cite{Sims1994aa}]
  If $S$ is a semigroup, $A\subseteq S$, and $w\in A ^ +$ is
  reduced for $S$, then every non-empty subword of $w$ is reduced. 
\end{prop}

Throughout this paper, we refer to the semigroup $U$ as the
\textit{universe}, and we let $S$ be a subsemigroup of $U$ represented by a
finite set $A$ of generators. The algorithms described herein require that we
can: 
\begin{itemize}
  \item compute the product of two elements in $U$; and
  \item test the equality of elements in $U$.
\end{itemize}

In each of the main algorithms presented in this paper, rather than obtaining
the actual elements of the subsemigroup $S$ we obtain reduced words
representing the elements of $S$. This is solely for the simplicity of the
exposition in this paper, and is equivalent to using the actual elements of $S$
by the following proposition.

\begin{prop}[cf. Proposition 2.3 in \cite{Froidure1997aa}]
   If $S$ is a semigroup, $A\subseteq S$, $\nu: A ^ + \to S$ is the unique
   homomorphism extending the inclusion of $A$ into $S$, then the set $R =
   \set{\overline{w}}{w\in A ^*}$ with the operation defined by $u\cdot v =
   \overline{uv}$ is a semigroup isomorphic to $S$.
\end{prop} 


\section{The Froidure-Pin Algorithm}\label{section-fropin}

In this section we describe a version of the Froidure-Pin Algorithm from
\cite{Froidure1997aa} and prove that it is valid.

Throughout this section, let $U$ be any semigroup, let $S$ be a subsemigroup of $U$
generated by $A\subseteq U$ where $1_U \in A$, and let $\nu: A ^+ \to S$ be the
unique homomorphism extending the inclusion function of $A$ into $S$.  Since
$\nu(a) = a$ for all $a\in A$, we can compute
$\nu(s)$ for any $s\in A ^ +$ by computing products of elements in $S$.

We require functions $f, l: A^+ \to A$ and  $p, s: A^+ \to A^*$
defined as follows: 
\begin{itemize}
  \item 
    if $w\in A^+$ and $w = au$ for some $a\in A$ and $u\in A^*$, then 
    $f(w) = a$ and $s(w) = u$, i.e. $f(w)$ is the first letter of $w$ and
    $s(w)$ is the suffix of $w$ with length $|w| - 1$;

  \item 
    if $w\in A^+$ and $w = vb$ for some $b\in A$ and $v\in A^*$, then 
    $l(w) = b$ and $p(w) = v$, i.e. $l(w)$ is the last letter of $w$ and
    $p(w)$ is the prefix of $w$ with length $|w| - 1$. 

\end{itemize}
Note that if $w\in A$, then both $p(w)$ and $s(w)$ equal the empty word. 

Next, we give formal definition of the input and output of the version of
the Froidure-Pin Algorithm described in this section. A less formal discussion
of the definition can be found below. 


\begin{de}
  \label{de-data-structure}
  The input of our version of the Froidure-Pin Algorithm is a
  \textit{snapshot of $S$} which is a tuple $(A, Y, K, B, \phi)$
  where:

  \begin{enumerate}[(a)]
      
    \item $A = \{a_1, \ldots, a_r\}$ is a finite collection of
      generators for $S$ where $a_1 < \cdots < a_r$ for some $r\in \N$;
    
    \item $A \subseteq Y = \{y_1, \ldots, y_N\}$ is a
      collection of reduced words for $S$,  $y_1 < y_2 < \cdots <
      y_N$, and if $x\in A^+$ is reduced, then either $x\in Y$ or $x > y_N$, 
      for some $N \in \N$;

    \item 
      $K\in \N$ and $1\leq K\leq |Y| + 1$;

    \item either  $B = \varnothing$ or $B = \{a_1, \ldots, a_s\}$ for some
          $1\leq s \leq r$; and 

    \item 
      if $K \leq |S|$, then $\phi$ is a function from 
      $$\big(A\times \{y_1, \ldots, y_{L}\}\big)\cup \big(\{y_1, \ldots,
      y_{K-1}\} \times A\big) \cup \big(\{y_K\}\times B\big)$$
      to $Y$ such that $\nu(\phi(u, v)) = \nu(uv)$ for all $u, v\in \dom(\phi)$
      and either $L = K - 1$ or $L < K - 1$ and $L$ is the largest value such
      that $|y_{L}| < |y_{K - 1}|$.
  \end{enumerate}
\end{de}
Note that by part (e) of Definition~\ref{de-data-structure}, $\phi(u,v) \leq
uv$ for all $u,v\in\dom(\phi)$ since $\phi(u, v)\in Y$ is reduced and both $uv$
and $\phi(u, v)$ represent the same element of $S$. If in parts (c) and (e) of
Definition~\ref{de-data-structure}, $L = K - 1$ and $B = A$, respectively, then
the snapshot $(A, Y, K, B, \phi)$ could be written as the snapshot $(A, Y, K +
1, \varnothing, \phi)$ instead. We allow both formulations so that
Algorithm~\ref{algorithm-update} (\textsf{Update}) is more straightforward. 


Roughly speaking, in Definition~\ref{de-data-structure}, part (b) say that $Y$
consists of all the reduced words for $S$ which are less than or equal to
$y_N$. The set $Y$ consists of those elements of $S$ that have been found so
far in our enumeration. The $K$ in part (c) is the index of the next element in
$Y$ that will be multiplied on the right by some of the generators $A$. More
precisely, $y_K$ will be multiplied by the least element in $A\setminus B$
where $B$ is given in part (d).  The function $\phi$ in part (e) represents a
subgraph of the right and left Cayley graphs of $S$ with respect to $A$.  More
specifically, if $B \not = A$, then $K$ is the index of the first element in
$Y$ where not all of the right multiples $y_Ka$ where $a\in A$ are yet known.
Similarly, $L$ is the index of the last element in $Y$ where all of the left
multiples of $y_L$ are known.  Condition (e) indicates that if any right
multiples of a reduced word $w$ for $S$ are known, then the left multiples of
all reduced words of length at most $|w| - 1$ are also known.  The condition
also allows the left multiples of all reduced words of length $|w|$ to be
known.

The output of our version of the Froidure-Pin Algorithm is another 
snapshot of the above type where: the output value of $K$ is at least the
input value; the output value of $Y$ contains the input value as a subset; and
the output function $\phi$ is an extension of the input function.  In this way,
we say that the output snapshot \textit{extends} the input data
structure. In other words, if $K \leq |S|$ or $Y \not =S$, then the output
snapshot contains more edges in the right Cayley graph than the input snapshot,
and may contain more elements of $S$, and more edges in the left Cayley graph.

The parameters $|Y|$ and $K$ quantify the state of the Froidure-Pin Algorithm,
in the sense that the minimum values are $|Y| = |A|$ and $K = 1$, and the
snapshot is \textit{complete} when $K = |S| + 1 = |Y| + 1$, which means that $Y
= S$ and all of the edeges of the left and right Cayley graphs of $S$ are
known.  We say that a snapshot of $S$ is \textit{incomplete} if it is not
complete.  If desirable the Froidure-Pin Algorithm can be halted before the
output snapshot is complete (when $K \leq |S|$), and subsequently continued and
halted, any number of times. Such an approach might be desirable, for example,
when testing if $u\in U$ belongs to $S$, we need only run the Froidure-Pin
Algorithm until $u$ is found, and not until $K = |S| + 1$ unless $u\not \in S$. 

The minimal snapshot required by the Froidure-Pin Algorithm for $S =
\genset{A}$ is:
\begin{equation}\label{eq-minimal}
  (A, A, 1, \varnothing, \varnothing);
\end{equation}
where the only known elements are the generators, and no left nor right
multiples of any elements are known. 

If $(A, Y, K, B, \phi)$ is a snapshot of the semigroup $S$, then $|Y|$ is
referred to as its \textit{size}, its \textit{elements} are the elements of
$Y$, and $x \in S$ \textit{belongs} to the snapshot if the reduced word
corresponding to $x$ belongs to $Y$, i.e.\ $x = \nu(w)$ for some $w \in Y$. 


Algorithm~\ref{algorithm-update} (\textsf{Update}) is an essential step in the
Algorithm~\ref{algorithm-froidure-pin} (\textsf{Froidure-Pin}) and we will
reuse (\textsf{Update}) in Algorithm~\ref{algorithm-closure}
(\textsf{Closure}).  Roughly speaking, \textsf{Update} adds the first unknown
right multiple to an incomplete snapshot. In \textsf{Update} a multiplication
is only performed if it is not possible to deduce the value from existing
products in the input snapshot. Deducing the value of a right multiple is a
constant time operation, whereas performing the multiplication  almost always
has higher complexity. For example, matrix multiplication is at least quadratic
in the number of rows.  This is the part of the Froidure-Pin Algorithm that is
responsible for its relatively good performance.  


\begin{algorithm}
\SetKwComment{Comment}{}{}
\DontPrintSemicolon
  \KwIn{An incomplete snapshot $(A, Y, K, B, \phi)$ for a semigroup $S$ where $B
        \not = A$}
  \KwOut{A snapshot extending $(A, Y, K, B, \phi)$ that
        contains $y_Ka$ where $a = \min(A\setminus B)$}
  \uIf(\Comment*[f]{$s(y_K)a$ is not reduced, and so $y_Ka$ is not reduced}){
  $\phi(s(y_K), a) < s(y_K)a$ } {   
    $\phi(s(y_K), a) = y_i$ for some $y_i\in Y$ \;
    $\phi(y_K, a) := \phi(\phi(f(y_K), p(y_i)), l(y_i))$\; 
  }
  \uElseIf(\Comment*[f]{$s(y_K)a$ is reduced but $y_Ka$ is not}){$\nu(y_Ka)
  = \nu(y_i)$ for some $i
  \leq |Y|$} {
    $\phi(y_K, a) := y_i$\;
  }
  \Else(\Comment*[f]{$y_Ka$ is reduced}){
    \label{final-1}
      $y_{|Y| + 1} := y_Ka$ and $Y \gets Y \cup \{y_{|Y| + 1}\}$ \;
    \label{final-2} 
      $\phi(y_K, a) := y_{|Y| + 1}$ \;
  }
  \Return $(A, Y, K, B \cup \{a\}, \phi)$
  \caption{\textsf{Update}: adds the first unknown right multiple to an
  incomplete snapshot}\label{algorithm-update}
\end{algorithm}

Recall that since $\nu(a) = a$ for all $a\in A$ and $\nu$ is a homomorphism,
the value of $\nu(w)\in S$ can be determined for any $w\in A ^ +$ by computing 
products in $S$. In practice, we perform the single multiplication 
$\nu(y_K)\ \nu(a) = \nu(y_Ka)$.

\begin{lem}\label{lem-fropin-1}
  If  $(A, Y, K, B, \phi)$ is a snapshot of a semigroup $S$, $B\not = A$, and
  $a$ is the least generator in $A\setminus B$, then
  Algorithm~\ref{algorithm-update} (\textsf{Update}) returns a snapshot of $S$
  containing $y_Ka$. 
\end{lem}
\begin{proof}
  There are three cases to consider: 
  \begin{enumerate}[(a)]
    \item
      $s(y_K)a$ is not reduced
    \item 
      $\nu(y_Ka) = \nu(y_i)$ for some $i \leq N$
    \item 
      neither (a) nor (b) holds.
  \end{enumerate}
  If (a) or (b) holds, then the only component of the snapshot that is modified
  is $\phi$, and so we must verify that $\phi$ is well-defined, and satisfies
  Definition~\ref{de-data-structure}(e). In the case when (c) holds, we also 
  need to show that Definition~\ref{de-data-structure}(b) holds.

  \textbf{(a).}
  In this case, $\phi(s(y_K), a)\in Y$ is defined because $s(y_K)\in Y$ and
  $s(y_K) < y_K$. Hence there exists $i \leq N$ such that $\phi(s(y_K), a) =
  y_i$.  Since $p(y_i)l(y_i) = y_i < s(y_K)a$, by
  Proposition~\ref{prop-fropin-1}(b), $p(y_i) \leq s(y_K)$ and $s(y_K) < y_K$
  since $|s(y_K)| < |y_K|$.  Hence $\phi(f(y_K), p(y_i))$ is defined and
  $$\phi(f(y_K), p(y_i)) \leq f(y_K)p(y_i) < f(y_K)s(y_K) = y_K.$$

  By the definition of $\phi$ and since $\nu$ is a homomorphism, 
  \begin{eqnarray*}
    \nu(y_Ka) & = & \nu(f(y_K) s(y_K) a) =\nu(f(y_K)) \nu(s(y_K) a)
    =
  \nu(f(y_K))\nu(\phi(s(y_K), a)) \\
    & = & \nu(f(y_K))\nu(y_i) = \nu(f(y_K) y_i) 
  = \nu(f(y_K) p(y_i)) \nu(l(y_i)) \\
    & = & \nu(\phi(f(y_K), p(y_i)))\nu(l(y_i)) 
  = \nu(\phi(f(y_K), p(y_i))l(y_i)) \\ 
    & = & \nu(\phi(\phi(f(y_K), p(y_i)),l(y_i))).
  \end{eqnarray*}
  Hence if $\phi(y_K, a) := \phi(\phi(f(y_K), p(y_i)), l(y_i))$, then $\phi$
  continues to satisfy Definition~\ref{de-data-structure}(e).

  \textbf{(b).}
  By the assumption of this case, $\nu(\phi(y_K, a)) = \nu(y_i) = \nu(y_Ka)$,
  and Definition~\ref{de-data-structure}(e) continues to hold.

  \textbf{(c).}
  In this case, $y_Ka$ is reduced, and $y_Ka \notin Y$. Hence $y_Ka > \max Y$
  by Definition~\ref{de-data-structure}(b), and so $Y \cup \{y_{|Y| + 1}\}$
  where $y_{|Y|+1} = y_Ka$ satisfies the first part of
  Definition~\ref{de-data-structure}(b). If $w\in A^+$ is reduced and $w < y_Ka
  = y_{|Y|+1}$, then $w = p(w) l(w) < y_Ka$ and so either $p(w) < y_K$  or $l(w) <
  a$ by Proposition~\ref{prop-fropin-1}. In both cases, $(p(w), l(w)) \in
  \dom(\phi)$, and so $w =\phi(p(w), l(w)) \in Y$ and so
  Definition~\ref{de-data-structure}(b) holds. 
  
  By the assumption of this case $\nu(y_Ka) \not= \nu(y_i)$ for any $y_i\in Y$,
  and so, in particular, $\overline{y_Ka}\not \in Y$.  On the other hand,
  $\overline{y_Ka} \leq y_Ka$ and $\overline{y_Ka} < y_Ka$ implies
  $\overline{y_Ka}\in Y$ from the previous paragraph. Hence
  $y_Ka=\overline{y_Ka}$ is reduced.  Finally, $\nu(\phi(y_K, a)) =
  \nu(y_{|Y|+1}) = \nu(y_Ka)$ by definition and so
  Definition~\ref{de-data-structure}(e) holds.
\end{proof}

Next we state the Froidure-Pin Algorithm, a proof that the algorithm is
valid follows from Lemmas~\ref{lem-fropin-1} and~\ref{lem-fropin-2}. 
\begin{algorithm}
\SetKwComment{Comment}{}{}
\DontPrintSemicolon
  \KwIn{A snapshot $(A, Y, K, \varnothing, \phi)$ for a semigroup $S$ and a
  limit $M\in \N$}
  \KwOut{A snapshot of $S$ which extends
      $(A, Y, K, \varnothing, \phi)$ and with size at least $\min\{M, |S|\}$}
  \While{$K \leq |Y|$ and $|Y| < M$ \label{line-1} }{
    $c := |y_K|$ \;
    \While{$K \leq |Y|$ and $|y_K| = c$ and $|Y| < M$\label{inner-while}} {
      $B := \varnothing$ \;
      \While(\Comment*[f]{loop over the generators $A$ in (short-lex) order}){$B\setminus A \not= \varnothing$} {
        $(A, Y, K, B, \phi) \gets$ \textsf{Update}$(A, Y, K, B, \phi)$
        \label{end-for} \;
      }
      $K \gets K + 1$\;
    }
    \If{$K > |Y|$ or $|y_K| > c$ \label{if-clause}} {
       $L = \max\set{i \in \N}{|y_i| < c}$ \label{fropin-define-L}\;
      \For(\Comment*[f]{indices of words of length one less than $y_K$}){$i\in \{L + 1, \ldots, K - 1\}$} { 
        \For(\Comment*[f]{extend $\phi$ so that $A\times
                  \{y_i\}\subseteq \dom(\phi)$}) {$a\in A$} { 
          $\phi(a, y_i) := \phi(\phi(a, p(y_i)), l(y_i))$ \label{left} 
        }
      }
    }
  }
  \Return $(A, Y, K, \varnothing, \phi)$
  \caption{\textsf{FroidurePin}: enumerate at least $\min\{M, |S|\}$ elements
  of a semigroup $S$.} 
  \label{algorithm-froidure-pin}
\end{algorithm}

Note that both the input, and output, of Algorithm~\ref{algorithm-froidure-pin}
(\textsf{FroidurePin}) has fourth component equal to $\varnothing$, and as such
it would appear to be unnecessary. However, it is used in the definition of a
snapshot so that we can succinctly describe the output of \textsf{Closure}.
\textsf{FroidurePin} could be modified to return a snapshot where the fourth
component was not empty, but for the sake of relative simplicity we opted not
to allow this.

\begin{lem}\label{lem-fropin-2}
  If $(A, Y, K, \varnothing, \phi)$ is a snapshot of a semigroup $S$ and
  $M\in \N$, then Algorithm~\ref{algorithm-froidure-pin} (\textsf{FroidurePin})
  returns a snapshot for $S$ with at least $\min\{M, |S|\}$ elements.
\end{lem}
\begin{proof}
  We may suppose that $K\leq |Y|$ and $|Y| < M$, since otherwise
  \textsf{FroidurePin} does nothing. 

  By Lemma~\ref{lem-fropin-1}, after applying 
  \textsf{Update} to the input snapshot and every $a\in A$, in
  line~\ref{end-for}, $(A, Y, K, A, \phi)$ is a snapshot of $S$ containing
  $y_KA$.  At this point, if $K \leq |Y|$ and $|y_{K + 1}| = |y_K|$, then we
  continue the while-loop starting in line~\ref{inner-while}.  When we arrive
  in line~\ref{if-clause} one of the following holds: $K > |Y|$, $|Y|\geq M$,
  or $|y_K| > c = |y_{K - 1}|$.
  
  If the condition in line~\ref{if-clause} is not satisfied, then the condition
  in line~\ref{inner-while} is false because $|Y| \geq M$. Hence the condition
  in line~\ref{line-1} is also false, and so $(A, Y, K, A, \phi)$ is returned
  in this case. Since Algorithmn~\ref{algorithm-update} (\textsf{Update})
  returns a snapshot, $(A, Y, K, A, \phi)$ could only fail to be a snapshot
  because the value of $K$ was increased. In other words, the tuple $(A, Y, K,
  A, \phi)$ satisfies Definition~\ref{de-data-structure} (a) to (d) and
  $\nu(\phi(u,v)) = \nu(uv)$ for all $(u, v) \in \dom(\phi)$.
  Since the condition in line~\ref{if-clause} is not satisfied,
  $|y_K| = |y_{K-1}| = c$, and so $(A, Y, K, A, \phi)$ continues to satisfy
  Definition~\ref{de-data-structure}(e), which is hence a snapshot of $S$,
  as required.
 
  If the condition in line~\ref{if-clause} is satisfied, then 
  either $K > |Y|$; or $K \leq |Y|$ and $|y_{K}| > |y_{K-1}| = c$.
  In either case, the tuple $(A, Y, K, B, \phi)$ satisfies
  Definition~\ref{de-data-structure}(a) to (d), but may fail to satisfy part (e),
  since $\phi$ may not be defined on $A \times \{y_{L + 1}, \ldots, y_{K-1}\}$,
  where $L\in \N$ is the maximum value such that $|y_L| < |y_{K - 1}|$ (as
  defined in line~\ref{fropin-define-L}).  The
  only component of $(A, Y, K, B, \phi)$ that is modified within the if-clause
  is $\phi$.  Hence by the end of the if-clause  $(A, Y, K, B, \phi)$ is a
  snapshot for $S$, provided that $\phi(a, y_i)$ is well-defined for all $i\in
  \{L + 1, \ldots, K - 1\}$ and $\nu(\phi(a, y_i)) = \nu(ay_i)$ for all $a\in A$.

  Let $i \in \{L + 1, \ldots, K - 1\}$.
  Since $|p(y_i)| = |y_i| - 1$ and $|y_i| = |y_{K - 1}|$, $\phi(a, p(y_i))$ is
  defined and $|\phi(a, p(y_i))| \leq |ap(y_i)| = |y_{K-1}|$, because $\phi(a,
  p(y_i)) \leq ap(y_i)$.
  Since $y_{K-1}$ is the largest reduced word of length
  $|y_{K-1}|$, there exists $j < K$ such that $\phi(a, p(y_i)) = y_j$. By
  definition, $(y_j, x)\in \dom(\phi)$ for all $x\in A$.
  In particular, $\phi(y_j, l(y_i)) = \phi(\phi(a, p(y_i)), l(y_i))$ is
  defined, and so the assignment in line~\ref{left} is valid.
  
  Let $a\in A$ and $i\in \{L + 1, \ldots, K - 1\}$ be arbitrary. Then 
  \begin{eqnarray*}
    \nu(\phi(a, y_i)) & = & \nu(\phi(\phi(a, p(y_i)), l(y_i))) = 
     \nu(\phi(a, p(y_i))l(y_i)) = \nu(\phi(a, p(y_i)))\ \nu(l(y_i)) \\
     & = & \nu(ap(y_i))\ \nu(l(y_i)) =  \nu(ap(y_i)l(y_i)) \\
     & =  & \nu(ay_i),
  \end{eqnarray*}
  as required.
  
  Finally, the algorithm halts if $|K| > |Y|$ in which case the snapshot
  contains $|S|$ elements, or if $|Y| \geq M$. In either case, $|Y| \geq
  \min\{M, |S|\}$, as required.
\end{proof}

\begin{cor}
  If $(A, Y, K, \varnothing, \phi)$ is a snapshot of a semigroup $S$ and
  $M\in \N$ is such that $M\geq |S|$, then
  Algorithm~\ref{algorithm-froidure-pin} (\textsf{FroidurePin})
  returns $(A, R, |S| + 1, \varnothing,
  \phi)$ where $R$ is the set of all reduced words for elements of $S$
  and $\dom(\phi) = (A \times R) \cup (R \times A)$.
\end{cor}
\begin{proof}
   Suppose that the snapshot returned by \textsf{FroidurePin} is $(A, Y, |S| +
   1, \varnothing, \phi)$. 

   Under the assumptions of the statement, the last iteration of while-loop
   starting in line~\ref{inner-while} terminates when $K = |Y| + 1$. Hence the
   condition of the if-clause in line~\ref{if-clause} is satisfied, and so
   $\dom(\phi) =  (A \times Y) \cup (Y \times A)$.

   Assume that there exists a reduced word $w\in A^+$ such that $w\not \in Y$.
   Then we may assume without loss of generality that $w$ is the minimum such
   reduced word. Hence $p(w)$ is a reduced word and $p(w) < w$ and so $p(w) \in
   Y$. But then $\phi(p(w), l(w))$ is defined, and so $w = \phi(p(w), l(w)) \in
   Y$, a contradiction. Hence $Y = R$.
\end{proof}


\section{The closure of a semigroup and some elements}\label{section-closure}

In this section we give the first of the two new algorithms in this paper.
Given a snapshot of a semigroup $S = \genset{A} \leq U$ and some
additional generators $X\subseteq U$, this algorithm returns a snapshot
for $T = \genset{A, X}$. 
If $\nu: (A \cup X) ^ + \to T$ is the unique homomorphism extending the
inclusion map from $A \cup X$ into $T$, then $\nu$ restricted to $A^+$ is the
unique homomorphism extending the inclusion map from $A$ into $S$. Hence we
only require the notation $\nu: (A \cup X) ^ + \to T$.

The purpose of Algorithm~\ref{algorithm-closure} (\textsf{Closure}) is to avoid
multiplying elements in the existing snapshot for $S$, by the generators $A$,
in the creation of the snapshot for $T$ wherever such products are already
known. The principal complication is that the introduction of new generators
can change the reduced word representing a given element $s$ of $S$. The new
generating set $A\cup X$ may allow $s$ to be written as a shorter word than was
previously possible with $A$.  Suppose that $\{y_1, \ldots, y_{K_S}\}$ is the
set of of those elements in the original snapshot for $S$ whose right multiples
by the old generating set were known.  \textsf{Closure} terminates when the
right multiplies by all the generators, old and new, of every element in
$\{y_1, \ldots, y_{K_S}\}$ are known. In short \textsf{Closure} is a version of
\textsf{FroidurePin} where each of $K_S|A|$ products can be found in constant
time.  This works particularly well when the addition of new generators does
not increase the size of the semigroups substantially, or when the complexity
of multiplication is very high.

We prove that \textsf{Closure} is valid in
Lemma~\ref{lem-algorithm-closure-is-valid}.

\begin{algorithm}
\SetKwComment{Comment}{}{}
\DontPrintSemicolon
  \KwIn{A snapshot $(A, Y, K_S, \varnothing, \phi_S)$ for a semigroup $S\leq
  U$, and a collection of elements $X\subseteq U\setminus Y$}
  \KwOut{A snapshot $(A \cup X, Z, K_T, \varnothing,
     \phi_T)$ for $T = \genset{S, X}$ that contains $\{y_1, \ldots, y_{K_S}\}X$.}

      Define $A = \{z_1 = y_1, \ldots, z_m = y_m\}$ and $X = \{z_{m + 1}, \ldots, z_{m
      + n}\}$, $z_1 < z_2 < \cdots < z_{m + n}$ \;
      $Z := A\cup X$, $K_T := 1$, $\phi_T = \varnothing$
      \Comment*[r]{initialise the snapshot for $T$}
      Define $\lambda: A \to Z$ to be the identity function on $A$ 
      \Comment*[f]{$\lambda$ maps a reduced word for $S$ to a reduced word for $T$
      representing the same element, at this point $A$ is the set $\set{y\in Y}{\exists z\in Z,\
      \nu(y) = \nu(z)}$}\;
      \While(\Comment*[f]{there are reduced words in $Y$ not yet in the
      snapshot for $T$}){$\dom(\lambda) \not= Y$} {
        $c := |z_{K_T}|$\;
        \While{$\dom(\lambda) \not= Y$ and $|z_{K_T}| = c$
        \label{closure-inner-while} }{
        $ B := \varnothing$ \;
      \If{$\exists y_i\in Y$, $\nu(z_{K_T}) = \nu(y_i)$ and $i < K_S$
      \label{closure-if-1}} {
        \For(\Comment*[f]{loop over the old generators in (short-lex) order}){$a\in A$} {
          \uIf{$\phi_S(y_i, a)\in \dom(\lambda)$} {
            $\phi_T(z_{K_T}, a) := \lambda(\phi_S(y_i, a))$
            \label{mod-phi_T-1} \;
          } \Else {
             $\phi_T(z_{K_T}, a) := z_{K_T}a$ and $\lambda(\phi_S(y_i, a)) :=
             z_{K_T}a$ \label{lambda-mod-1}\;
             $z_{|Z| + 1} := z_Ka$ and $Z\gets Z\cup \{z_{K_T}a\}$ \;
          }
        }
        $B \gets A$\;
      }
      \While{$(A\cup X)\setminus B \not=\varnothing$}  {
          $a := \min\{(A\cup X)\setminus B\}$
          \label{a-is-defined}\;
         $(A\cup X, Z, K_T, B, \phi_T) \gets$ 
         \textsf{Update}$(A\cup X, Z, K_T, B, \phi_T)$\;

         \If{$\phi_T(z_{K_T}, a) = z_{K_T}a$ and $\nu(z_{K_T}a) = \nu(y_i)$ for some
         $y_i\in Y$}{
            $\lambda(y_i) := z_{K_T}a$ \;\label{lambda-mod-2}\;
         }
      }
      $K_T \gets K_T + 1$\;
    }
    \If{$K_T > |Z|$ or $|z_{K_T}| > c$} {
       $L = \max\set{i \in \N}{|z_i| < c}$ \;
      \For {$i\in \{L + 1, \ldots, K_T - 1\}$} {
        \For{$a\in A \cup X$} {
          $\phi_T(a, z_i) := \phi_T(\phi_T(a, p(z_i)), l(z_i))$ \;
        }
      }
    }
  }
  \caption{\textsf{Closure}: add additional generators to a snapshot for a
  semigroup} 
  \label{algorithm-closure}
\end{algorithm}

\begin{lem}\label{lem-algorithm-closure-is-valid}
  If $(A, Y, K_S, \varnothing, \phi_S)$ is a snapshot of a semigroup $S\leq U$
  and $X\subseteq U$, then Algorithm~\ref{algorithm-closure}
  (\textsf{Closure}) returns a snapshot of $T = \genset{X, A}$ that contains
  $\{y_1, \ldots, y_{K_S}\}X$, and hence also contains $Y$.
\end{lem}
\begin{proof}
  At the start of \textsf{Closure}, $(A \cup X, Z,
  K_T, B, \phi_T)$ is initialised as $(A\cup X, A\cup X, 1,
  \varnothing, \varnothing)$, which is the minimal snapshot of $T$ by
  \eqref{eq-minimal}.

  Additionally, $\lambda: A \to Z$ satisfies the following two conditions:
  \begin{enumerate}[\rm (a)]
    \item 
      $\dom(\lambda) = \set{y\in Y}{\exists z\in Z,\ \nu(y) = \nu(z)}$; and

    \item 
      $\nu(\lambda(y)) = \nu(y)$ for all $y\in \dom(\lambda)$.
  \end{enumerate}
  
  If $Y = \dom(\lambda)$, then the minimal snapshot of $T$ is returned,
  and there nothing to prove. So, suppose that $Y \not = \dom(\lambda)$.  We
  proved in Lemma~\ref{lem-fropin-1}, that Algorithm~\ref{algorithm-update}
  (\textsf{Update})
  returns a snapshot of $T$, given a snapshot of $T$. The data
  structure $(A \cup X, Z, K_T, B, \phi_T)$ is otherwise only modified within
  the while-loop starting on line~\ref{closure-inner-while}, in the case that
  there exists $y_i\in Y$ such that $\nu(z_{K_T}) = \nu(y_i)$ and $i < K_S$.
  Hence it suffices to verify that after performing the steps in the if-clause
  starting in line~\ref{closure-if-1} the tuple $(A \cup X, Z, K_T, B, \phi_T)$
  is still a snapshot of $T$. In order to do this, we use the properties
  of $\lambda$ given above. Hence we must also check that $\lambda$ continues
  to satisfy conditions (a) and (b) whenever it is modified (i.e.\ in
  lines~\ref{lambda-mod-1} and~\ref{lambda-mod-2}).

  Suppose that $\lambda$ satisfies conditions (a) and (b) above, and that
  there exists $y_i\in Y$ such that $\nu(z_{K_T}) = \nu(y_i)$ and $i < K_S$ and
  that $a\in A$. Since $i < K_S$, $\phi_S(y_i, a)$ is defined for all $a\in A$.
  
  If $\phi_S(y_i, a) \in \dom(\lambda)$, then in line~\ref{mod-phi_T-1} we define 
  $\phi_T(z_{K_T}, a) = \lambda(\phi_S(y_i, a))$. In this case, 
   $(A \cup X, Z, K_T, B, \phi_T)$ satisfies
   Definition~\ref{de-data-structure}(a) to (d) trivially and 
  $$\nu(\phi_T(z_{K_T}, a)) = \nu(\lambda(\phi_S(y_i, a))) =  \nu(\phi_S(y_i,
  a)) = \nu(y_ia) = \nu(z_{K_T}a).$$
  and so Definition~\ref{de-data-structure}(e) holds.
 
  If $\phi_S(y_i, a)\not\in \dom(\lambda)$, then we define $\phi_T(z_{K_T}, a)
  = z_{K_T}a$. Since $\nu(\phi_T(z_{K_T}, a)) = \nu(z_{K_T}a)$ by definition,
  it suffices to show that $z_{K_T}a$ is a reduced word for
  $T$. Suppose, seeking a contradiction, that $u\in (A\cup X) ^+$ is reduced,
  $\nu(u) = \nu(y_ia)$, and $u < z_{K_T}a$. Then either $p(u) < z_{K_T}$ or 
  $p(u) = z_{K_T}$ and $l(u) < a$ by Propostion~\ref{prop-fropin-1}(b). In
  either case, $\phi_T(p(u), l(u)) = u$ is defined and so $u\in Z$. Thus
  $\nu(\phi_S(y_i, a)) = \nu(y_ia) = \nu(u)$ and so $\phi_S(y_i, a) \in
  \dom(\lambda)$ by (a), which contradicts the assumption of this case. Hence
  $z_{K_T}a$ is reduced. We must verify conditions (a) and (b) on $\lambda$
  after defining $\lambda(\phi_S(y_i, a)) = z_{K_T}a \in Z$ in
  line~\ref{lambda-mod-1}. Condition (a) holds, since we extended both $Z$ and
  $\dom(\lambda)$ by a single value. Since $$\nu(\lambda(\phi_S(y_i, a))) =
  \nu(z_{K_T}a) = \nu(z_{K_T})\nu(a) = \nu(y_i)\nu(a) = \nu(y_ia) =
  \nu(\phi_S(y_i, a))$$ condition (b) also holds.

  The only other part of the algorithm where $\lambda$ is modified is
  line~\ref{lambda-mod-2}. Suppose that $a = \min((A\cup
  X)\setminus B)$ as
  defined in line~\ref{a-is-defined}. 
  If $\phi_T(z_{K_T}, a) = z_Ka$ and $\nu(z_{K_T})= \nu(y_i)$ for some $y_i\in Y$,
  then $z_Ka\in Z$ is reduced. 
  In this case, we define $\lambda(y_i) = z_{K_T}a$. Conditions (a) and (b)
  hold by the above argument.

  We have shown that within the while-loop starting on
  line~\ref{closure-inner-while}, the tuple $(A \cup X, Z, K_T, B, \phi_T)$
  satisfies Definition~\ref{de-data-structure}(a) to (d).  Additionally, we
  have shown that conditions (a) and (b) hold for $\lambda$. 
 
  The remainder of the proof follows by the argument given in the proof of
  Lemma~\ref{lem-fropin-2}.
\end{proof}


\subsection{Experimental results}

In this section we compare the performance of
Algorithm~\ref{algorithm-froidure-pin} (\textsf{FroidurePin}) and
Algorithm~\ref{algorithm-closure} (\textsf{Closure}). Further details about the
computations performed, including code that can be used to reproduce the data,
in this section can be found in~\cite{Mitchell2017ac}.

We note that in the implementation of
\textsf{Closure} in \libsemigroups~\cite{Mitchell2016ab}, the snapshot of $S=
\genset{A}$ is modified in-place to produce the data structure for $T=
\genset{A, X}$, and that none of the elements of $S$ need to be copied or moved
in memory during \textsf{Closure}.

Figure~\ref{fig-bmats-1},~\ref{fig-trans-1},~\ref{fig-bmats-2},
and~\ref{fig-trans-2} we plot a comparison of \textsf{FroidurePin} and
\textsf{Closure} for 300 examples of randomly generated semigroups.  For a
particular choice of $A$ and $X$, three computations were run:
\begin{enumerate}[(1)]
  \item \label{closure-benchmark-step-1}
    a snapshot for $\genset{A}$ was enumerated using
    \textsf{FroidurePin} until it
    contained $|\genset{A}|$ elements; we denote the time taken for this step
    by $t_1$.

  \item \label{closure-benchmark-step-2}
    \textsf{Closure} was performed on the snapshot obtained from
    \eqref{closure-benchmark-step-1} and the generators $X$. We denote the size
    of the snapshot obtained in this way by $M$ and the time taken for this
    step by $t_2$.

  \item \label{closure-benchmark-step-3}
    A snapshot for $\genset{A, X}$ was enumerated using \textsf{FroidurePin}
    until it contained $M$ (from \eqref{closure-benchmark-step-2}) elements.
    We denote the time taken for this step by $t_3$.
\end{enumerate}
A point on the $x$-axis of any graph in this section correspond to a single
choice of $A$ and $X$. The points on the $x$-axes are sorted in increasing
order according to the $M / \genset{A}$. In other words, $M / \genset{A}$
indicates the proportion of new elements generated in \textsf{Closure}.
The $y$-axis corresponds to time divided by $t_3$ from
\eqref{closure-benchmark-step-3}.  The triangles correspond to the mean
value of $t_2 / t_3$ taken over 3 separate runs, and the crosses
similarly correspond to $(t_1 + t_2) / t_3$. 

The values of $A$ and $X$ in Figure~\ref{fig-bmats-1} were chosen as follows:
\begin{itemize}
  \item 
    the collection $A$ was chosen to have size between $2$ and $30$ elements
    (uniformly at random) and to consist of $6\times
    6$ Boolean matrices, which were chosen uniformly at random from the space
    of all such Boolean matrices.
  \item 
    the collection $X$ was taken to consist of a single $6\times
    6$ Boolean matrices, again chosen uniformly at random.
\end{itemize}
The values of $A$ and $X$ in Figure~\ref{fig-trans-1} were chosen in a similar
fashion from the space of all transformations of degree $7$, where $|A| \in
\{2, \ldots, 8\}$ and $X$ was chosen so that $|X| = 1$. In
Figures~\ref{fig-bmats-2} and~\ref{fig-trans-2}, $A$ and $X$ were chosen in the
same way as collections of Boolean matrices, and transformations, with  $|A|,
|X| \in \{2, \ldots, 30\}$, and $|A|, |X| \in \{2, \ldots, 8\}$, respectively.
The semigroups in
Figures~\ref{fig-bmats-1},~\ref{fig-trans-1},~\ref{fig-bmats-2},
and~\ref{fig-trans-2} range in size from 1767 to 681973, with time to run
\textsf{FroidurePin} and \textsf{Closure} roughly in the range 10 to 1500
milliseconds, particular values chosen for $A$ and $X$ can be found in the
table below.  We rejected those choices of $A$ and $X$ where the time $t_1$ to
enumerate $\genset{A}$ using \textsf{FroidurePin} was less than 10
milliseconds, to reduce the influence of random noise in the resulting data.

The range of values for $|A|$, the dimensions of these matrices, and degrees of
these transformations were chosen so that the resulting semigroups were not too
large or small, and it was possible to run
\textsf{FroidurePin} on these
semigroups in a reasonable amount of time. 

\begin{figure}
  \centering
  \includegraphics[width=0.8\textwidth]{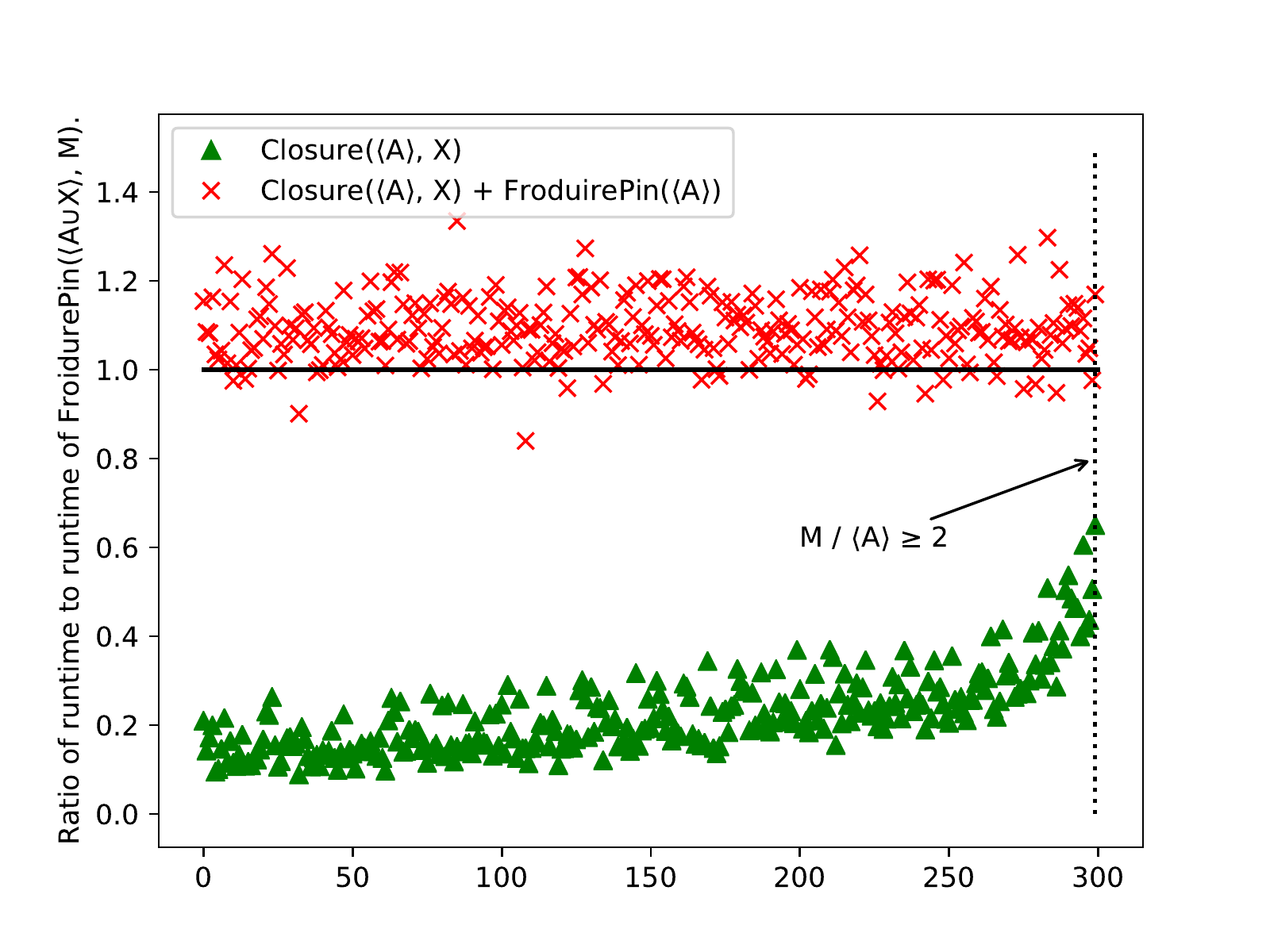}
  \caption{$A$ and $X$ consist of $6\times 6$ Boolean matrices, $|A|\in
  \{2,\ldots, 30\}$, $|X| = 1$, $M$ is the size of the snapshot for $\genset{A,
  X}$ returned by Algorithm~\ref{algorithm-closure} (\textsf{Closure}).}
  \label{fig-bmats-1}
\end{figure}

\begin{figure}
  \centering
  \includegraphics[width=0.8\textwidth]{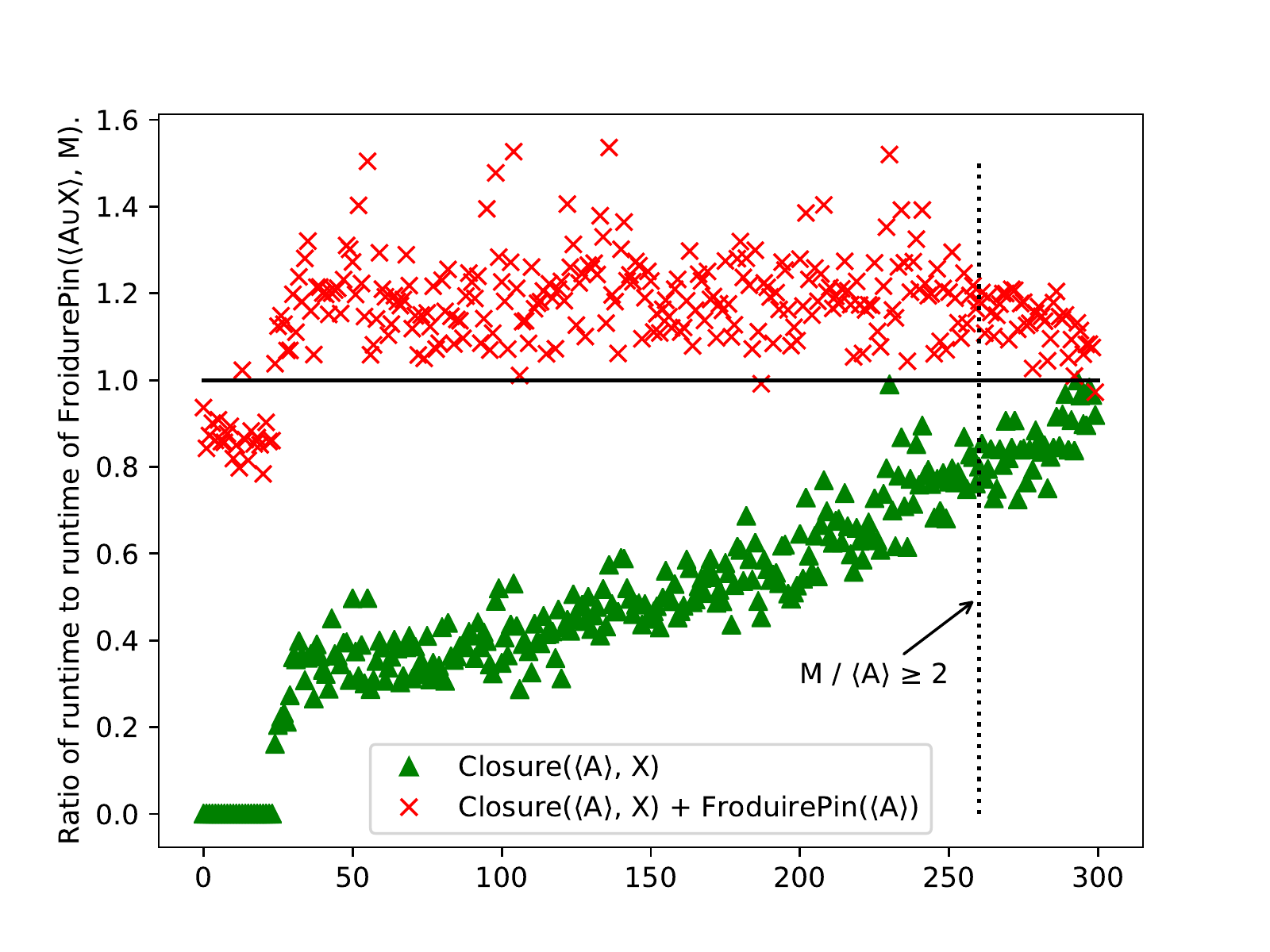}
  \caption{$A$ and $X$ consist of transformations of degree $7$, $|A|\in
  \{2,\ldots, 8\}$, $|X| = 1$, $M$ is the size of the snapshot for $\genset{A,
  X}$ returned by Algorithm~\ref{algorithm-closure} (\textsf{Closure}).}
  \label{fig-trans-1}
\end{figure}

\begin{figure}
  \centering
  \includegraphics[width=0.8\textwidth]{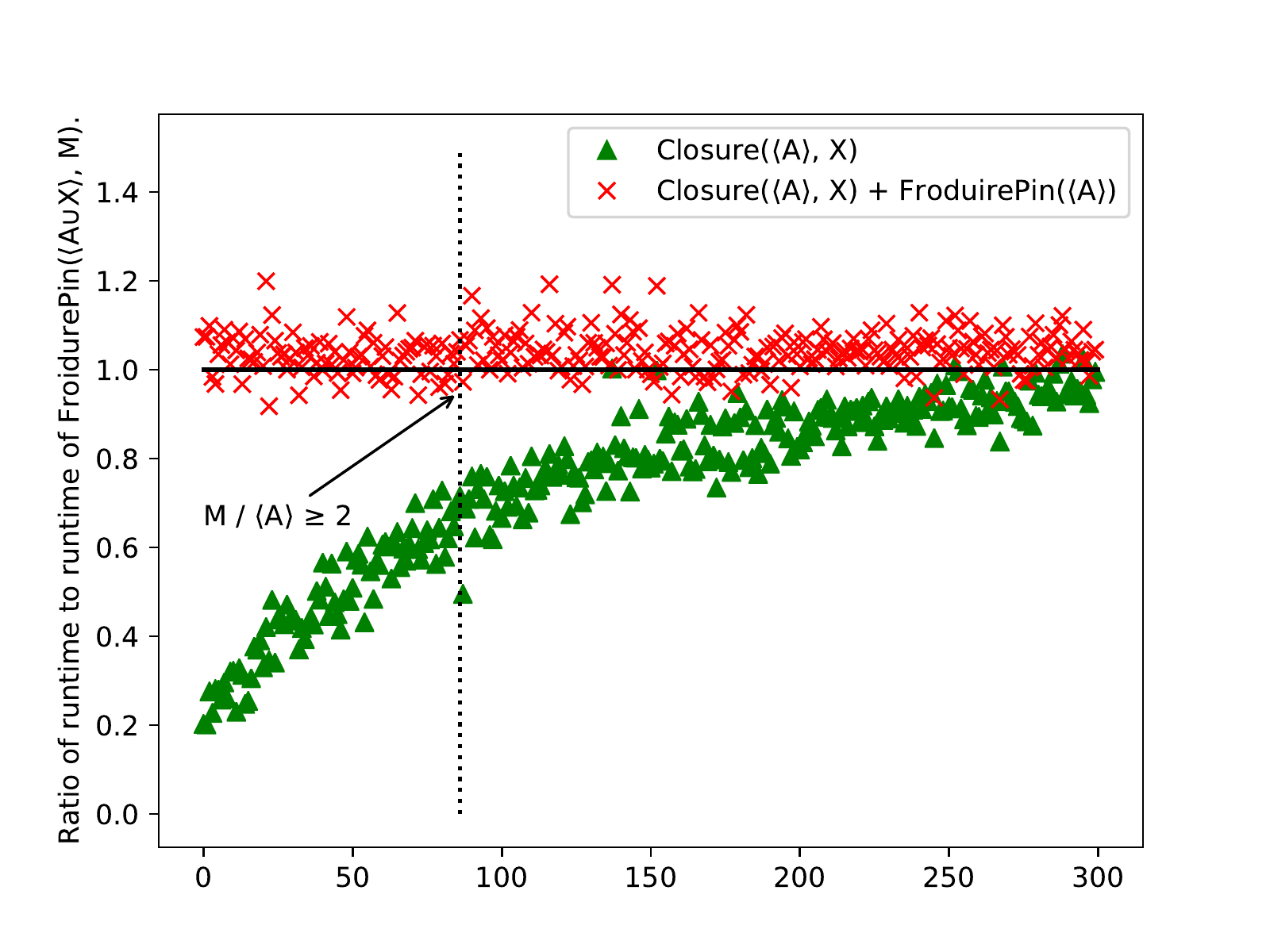}
  \caption{$A$ and $X$ consist of $6\times 6$ Boolean matrices, $|A|, |X| \in
  \{2,\ldots, 30\}$, $M$ is the size of the snapshot for $\genset{A,
  X}$ returned by Algorithm~\ref{algorithm-closure} (\textsf{Closure}).}
  \label{fig-bmats-2}
\end{figure}

\begin{figure}
  \centering
  \includegraphics[width=0.8\textwidth]{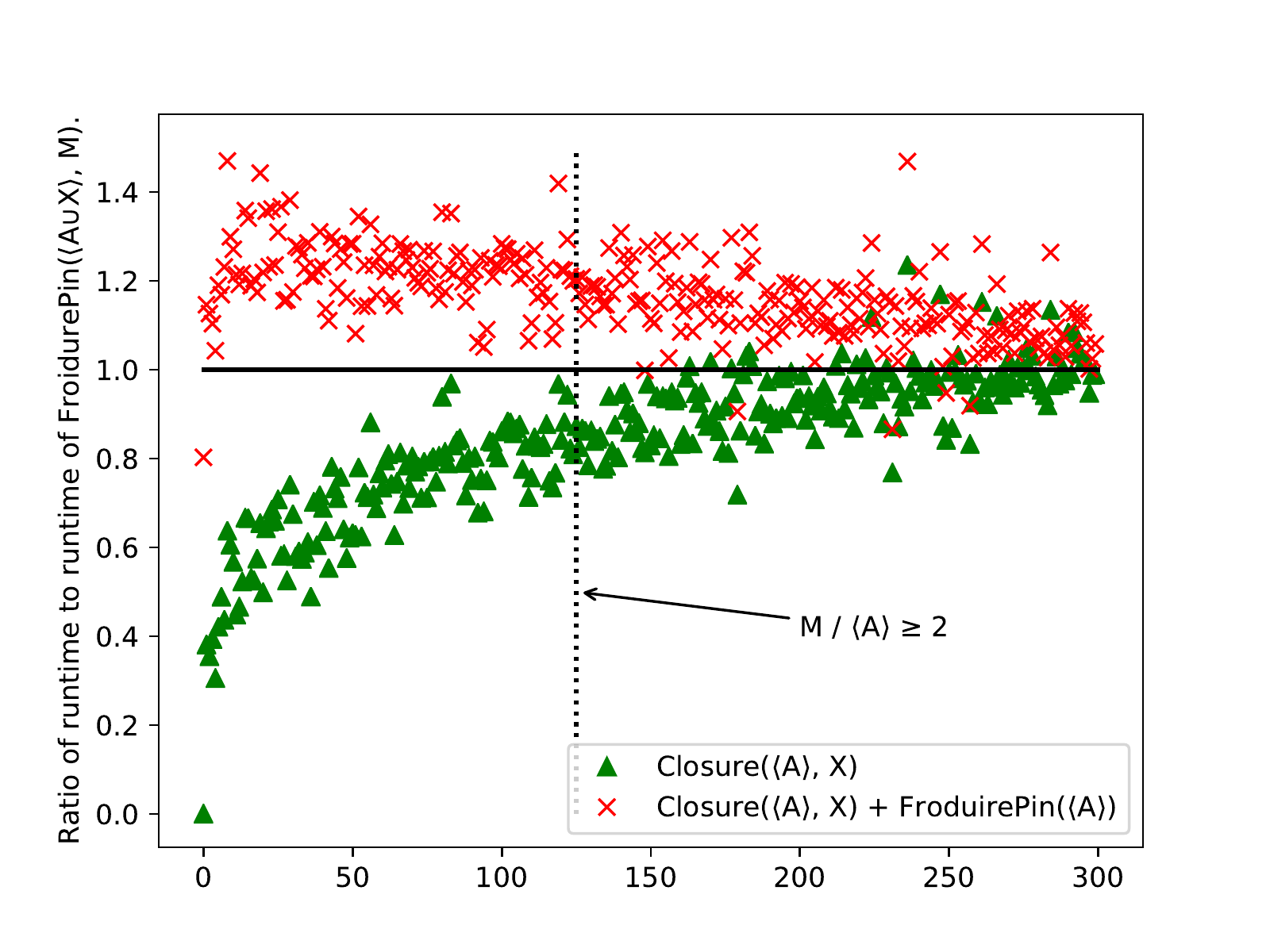}
  \caption{$A$ and $X$ consist of transformations of degree $7$, $|A|, |X|\in
  \{2,\ldots, 8\}$, $M$ is the size of the snapshot for $\genset{A,
  X}$ returned by Algorithm~\ref{algorithm-closure} (\textsf{Closure}).}
  \label{fig-trans-2}
\end{figure}

It is not surprising to observe that when $M / \genset{A}$ is large that there
is little benefit, or even a penalty, for running \textsf{Closure}. This due to
the fact that the run time $t_1$ of \textsf{FroidurePin} on $\genset{A}$ is
small in comparison to $t_2$ and $t_3$. In this case, in \textsf{Closure} applied
to $\genset{A}$ and $X$, only a relatively small proportion of multiplications
can be deduced, and so \textsf{Closure} is essentially doing the same
computation as \textsf{FroidurePin}, but with an additional overhead. 
It does appear from these examples, that
\textsf{Closure} is beneficial when $M / \genset{A} \leq 2$.


\section{A lock-free concurrent version of the Froidure-Pin Algorithm}
\label{section-parallel}

In this section we describe a version of the Froidure-Pin Algorithm that can be
applied concurrently in multiple distinct processes.  Roughly speaking, the
idea is to split a snapshot of a semigroup $S$ into fragments where an
algorithm similar to Algorithm~\ref{algorithm-froidure-pin}
(\textsf{FroidurePin}) can be applied to each fragment independently. Similar
to the Froidure-Pin Algorithm the value of a right multiple of an element with
a generator is deduced whenever possible. If such a right multiple cannot be
deduced, then it is stored in a queue until all fragments are synchronised. The
original Froidure-Pin Algorithm already offers a natural point for this
synchronisation to occur, when the right multiples of all words of a given
length have been determined. In the concurrent algorithm, the fragments are
synchronised at exactly this point.  In this way, the concurrent algorithm has
two phases: the first when right multiples are determined, and the second when
the fragments are synchronised.  Some of the deductions of products that are
made by the original Froidure-Pin Algorithm cannot be made in our concurrent
version. This happens if the information required to make a deduction is, or
will be, contained in a different fragment.  In order to avoid locking, a
fragment does not have access to any information in other fragments during the
first phase.  In the second phase, each fragment $F$ can only access the words
queued by all fragments for $F$. While this means that the concurrent algorithm
potentially performs more actual multiplications of elements than the original,
in Section~\ref{subsect-parallel-experiment} we will see that these numbers of
multiplications are almost the same. Of course, we could have allowed the
fragments to communicate during the first phase but in our experiments this was
significantly slower. The evidence presented in
Section~\ref{subsect-parallel-experiment}
supports this, the time spent waiting for other fragments to provide
information outweighs the benefit of being able to deduce a relatively small
additional number of products. 

Throughout this section, we again let $U$ be any semigroup, let $S$ be a
subsemigroup of $U$ generated by $A\subseteq U$, and let $\nu: A ^+ \to S$ be
the unique homomorphism extending the inclusion map.  We also define $f, l: A^+
\to A$ and $p, s: A^+ \to A^*$ to be the functions defined at the start of
Section~\ref{section-fropin}.

The following is the precise definition of a fragment required for our
concurrent algorithm.

\begin{de}
  \label{de-fragment}
  A \textit{fragment for $S$} is a tuple $(A, Y, K, \phi)$ where:

  \begin{enumerate}[(a)]
      
    \item $A = \{a_1, \ldots, a_r\}$ is a finite collection of generators for
      $S$ where $a_1 < \cdots < a_r$;

    \item $Y = \{y_1, \ldots, y_N\}$ is a
      collection of reduced words for $S$ and $y_1 < y_2 < \cdots < y_N$;

    \item 
      $1\leq K\leq |Y| + 1$ and if $K = 1$, then $Y = A$, or if $K > 1$, then
      $y_{K - 1}= \max\set{y_i\in Y}{|y_i| = |y_{K -1}|}$;

    \item 
      if  $R\subseteq A ^+$ is the set of all reduced words for elements of
      $S$, then 
      $$\phi: \big(A\times \{y_1, \ldots, y_{L}\}\big)\cup \big(\{y_1, \ldots,
      y_{K-1}\} \times A\big) \to R$$
      where $\nu(\phi(u, v)) = \nu(uv)$ for all $u, v\in \dom(\phi)$ and 
      either:
      \begin{enumerate}[(i)] 
        \item $L = K - 1$; or
        \item $L < K - 1$ and $L$ is the largest value such
          that $|y_{L}| < |y_{K - 1}|$.
      \end{enumerate}
  \end{enumerate}
\end{de}

Note that in part (c) the condition that $y_{K - 1}= \max\set{y_i\in Y}{|y_i| =
|y_{K -1}|}$ implies that either $y_{K}$ does not exist or $|y_K| > |y_{K
-1}|$.  We refer to the size, elements, extension, and so on of a fragment for
a semigroup $S$, in the same way as we did for the snapshots of $S$. 

We suppose throughout this section that however $\phi$ in
Definition~\ref{de-fragment} is actually implemented, it supports
concurrent reads of any particular value $\phi(x, y)$ when $(x,y) \in
\dom(\phi)$ and that it is possible to define $\phi(x, y)$ for any $(x, y)\not
\in \dom(\phi)$ concurrent with any read of $\phi(x', y')$ for $(x', y')\in
\dom(\phi)$. Of course, the difficulty is that we do not necessarily know
whether $(x, y)$ belongs to $\dom(\phi)$ or not. If
one process is defining $\phi(x, y)$ and another is 
trying to read $\phi(x, y)$, it is not possible to determine whether $(x,y)$
belongs to $\dom(\phi)$ without locking, which we want to avoid. We avoid this
in Algorithm~\ref{algorithm-parallel} (\textsf{ConcurrentFroidurePin}) because
$\phi$ is defined for those $(w, a)$ and $(a, w)$ where  $a\in A$ and $w$ is a
reduced word whose length is strictly less than  the minimum length of a word
some of whose right multiples are not known. 

The implementation in
\libsemigroups~\cite{Mitchell2016ab} represents $\phi$ as a pair of C++ Standard
Template Library vectors, which support this behaviour provided that no
reallocation occurs when we are defining $\phi(x,y)$ for $(x,y)\not\in
\dom(\phi)$. In \textsf{ConcurrentFroidurePin}, all of the reduced words
$w\in A ^ +$ of a given length are produced before any value of $\phi(w, a)$ or
$\phi(a, w)$ is defined, and so we can allocate enough memory to accommodate
these definitions and thereby guarantee that it is safe to read and write
values of $\phi(x,y)$ concurrently.

If $(A, A, 1, \varnothing, \varnothing)$ is a minimal snapshot
for $S$, we refer to the collection of fragments
$$(A, \set{a\in A}{b(a) = j}, 1, \varnothing)\quad j \in \{1,
\ldots, k\}$$
as a \textit{minimal collection of fragments} for $S$.

The next lemma describes when some fragments for $S$ can be assembled into a
snapshot of $S$. 

\begin{lem}\label{lem-union-fragments}
  Suppose that $(A, Y_1, K_1, \phi_1), \ldots, (A, Y_k, K_k, \phi_k)$ is a
  collection of fragments for $S$ which is either minimal or $Y_i = \{y_{i, 1},
  \ldots, y_{i, N_i}\}$ and the following hold:
  \begin{enumerate}[\rm (i)]
    \item 
      $Y_i\cap Y_j = \varnothing$ for all $i, j$, $i\not = j$;

    \item 
      for all reduced $w\in A^+$ either 
      $$w\in \bigcup_{i=1}^{k} Y_i\quad\text{or}\quad w > \max\bigcup_{i=1}^{k}
      Y_i;$$

    \item 
      if $K_j > 1$, then $|y_{j, K_j - 1}| = \max\set{|y|}{y\in Y_j,\ |y| \leq
      c}$ where $c = \max\set{|y_{i, K_i - 1}|}{1\leq i \leq k,\ K_i > 1}$

    \item 
      $\dom(\phi_i) = (A\times \{y_{i, 1}, \ldots, y_{i, K_i - 1}\}) \cup (\{y_{i,
      1}, \ldots, y_{i, K_i - 1}\} \times A)$ for all $i$; and

    \item \label{de-frag-v}
      $\im(\phi_i) \subseteq  \bigcup_{j=1}^{k} Y_j$ for all $i$.

  \end{enumerate}
  If $Y = \bigcup_{i=1}^{k}Y_i = \{y_1, \ldots, y_{N}\}$, $y_1 < \cdots <
  y_N$, $K\in \N$ is such that $y_{K - 1} = \max\set{y \in Y}{|y| =
  c}$, and $\phi = \bigcup_{i = 1} ^ k
  \phi_i$,  then $(A, Y, K, \varnothing, \phi)$ is a snapshot of $S$.
\end{lem}
\begin{proof} 
  If the collection of fragments is minimal, then clearly the union, as defined
  in the statement, is a minimal snapshot for $S$. 
  
  Suppose that the collection of fragments is not minimal. It follows that
  there exists $i\in \{1, \ldots, k\}$ such that $K_i > 1$, and so the value
  $c$ in part (iii) is well-defined.
  It is clear that Definition~\ref{de-data-structure}(a) to (d) are satisfied.
  It therefore suffices to show that part (e) of
  Definition~\ref{de-data-structure} holds. Part (i) of the assumption of this
  lemma implies that $\phi$ is well-defined.
  It suffices to verify that
  $$\dom(\phi) = (A\times \{y_{1}, \ldots, y_{K - 1}\}) \cup (\{y_{
      1}, \ldots, y_{K- 1}\} \times A),\quad \im(\phi) \subseteq Y,\quad
      \text{and}\quad \nu(\phi(u, v)) = \nu(uv)$$
  for all $u, v\in \dom(\phi)$. That  $\im(\phi) \subseteq Y$ follows
  from~\eqref{de-frag-v}. If $(u,v) \in \dom(\phi)$, then 
  $(u, v)\in\dom(\phi_i)$ for some $i$, and so 
  $\nu(\phi(u, v)) = \nu(\phi_i(u, v)) = \nu(uv)$.
 
  Let $(a, y_i) \in A\times \{y_1, \ldots, y_{K - 1}\}$. Then $y_i\in Y_j$ for
  some $j$ and so $y_i = y_{j,t}$ for some $t$. But $|y_{j, t}| = |y_i| \leq
  |y_{j, K_j - 1}|$ by the definition of $K$ and part (iii) of the assumption
  of this lemma. Hence either $|y_{j, t}| <  |y_{j, K_j - 1}|$ and so $y_{j, t}
  < y_{j, K_j - 1}$, or $|y_{j, t}| =  |y_{j, K_j - 1}|$ and, by
  Definition~\ref{de-fragment}(c), $y_{j, t} < y_{j, K_j - 1}$. In either
  case, $t \leq K_j - 1$, and so $(a, y_i) = (a, y_{j, t})\in \dom(\phi_j)
  \subseteq \dom(\phi)$.  If $(y_i, a) \in \{y_1, \ldots, y_{K - 1}\} \times
  A$, then $(y_i, a) \in \dom(\phi)$ by a similar argument. 

  If $(a, y_i)\in \dom(\phi)$, then $(a, y_i) \in \dom (\phi_j)$ for some $j$.
  Hence $y_i \in \{y_{j, 1}, \ldots, y_{j, K_j - 1}\}$. It follows that,  since
  $y_{K-1} = \max\set{y\in Y}{|y| = c}$, $y_{K -1} \geq 
  y_{j, K_j-1} \geq y_i$, and so $(a, y_i) \in A\times \{y_1, \ldots, y_{K -
  1}\}$. If $(y_i, a)\in \dom(\phi)$, then $(y_i, a) \in \{y_{i, 1}, \ldots, y_{i,
  K_i - 1}\} \times A$ by a similar argument. 
  Therefore $\dom(\phi) = (A\times \{y_{1}, \ldots, y_{K - 1}\}) \cup (\{y_{
      1}, \ldots, y_{K- 1}\} \times A)$, as required.
\end{proof}

In Algorithm~\ref{algorithm-apply-gens} (\textsf{ApplyGenerators}), and more
generally in our concurrent version of the Froidure-Pin Algorithm, we require a
method for assigning reduced words $w\in A^+$ that do not belong to any
existing fragment for $S$, to a particular fragment for $S$. If we want to
distribute $S$ into $k$ fragments, then we let $b: R := \set{w\in A^+}{w\
\text{is reduced for }S} \to \{1, \ldots, k\}$ be any function. Preferably so
that our algorithms are more efficient, $b$ should have the property that
$|b^{-1}(i)|$ is approximately equal to $|R|/ k$ for all $i$.  For example, we
might take a hash function for $\nu(w)$ modulo $k$, as the value of $b(w)$.  If
the number of fragments $k = 1$ or $b(w)$ is constant for all reduced words $w$
for $S$, then \textsf{ConcurrentFroidurePin} is just \textsf{FroidurePin} with
some extra overheads.

\begin{algorithm}
\SetKwComment{Comment}{}{}
\DontPrintSemicolon
  \KwIn{A collection of fragments $(A, Y_1, K_1, \phi_1), \ldots, (A, Y_k, K_k,
      \phi_k)$ for a semigroup $S$ satisfying the hypothesis of
      Lemma~\ref{lem-union-fragments} and $j\in \{1, \ldots, k\}$.}
  \KwOut{A set $Q_j = \set{(b(w), w)}{|w| = c + 1} \subseteq \{1,
  \ldots, k\} \times A ^ + \setminus (Y_1\cup\cdots \cup Y_k)$ where $c =
  \max\set{|y_{i, K_i - 1}|}{1\leq i \leq k}$ and the tuple $(A, Y_j, K_j,
  \phi_j)$.}
      $Q_j := \varnothing$ \Comment*[r]{$Q_j$ is a container for new words}
    \While{$K_j \leq |Y_j|$ and $|y_{j, K_j}| = c$ \label{algorithm-apply-gens-5}}
     {
      \For(\Comment*[f]{loop over the generators in (short-lex) order}){$a\in A$} {
        \If(\Comment*[f]{$s(y_{j, K_j})a$ is not reduced}) {$\phi_j(s(y_{j, K_j}), a) < s(y_{j, K_j})a$\label{algorithm-apply-gens-1}}{

            $y := \phi_{i}(s(y_{j, K_j}), a)$ where $s(y_{j, K_j}) \in Y_{i}$\;
            \label{algorithm-apply-gens-3}
            $w := \phi_{m}(f(y_{j, K_j}), p(y)) \in Y_n$ where $p(y) \in Y_{m}$\;
          \If{$n = j$ or $|w| < c$} {
            $\phi_j(y_{j, K_j}, a) := \phi_n(w, l(y))$ 
            \label{algorithm-apply-gens-4} \;
            continue \;
            }
          }

        \uIf{$\nu(y_{j, K_j}a) = \nu(y)$ for some $y \in Y_1 \cup \cdots \cup
        Y_k$} {
           $\phi_j(y_{j, K_j}, a) := y$ \;
        }\Else 
        { 
          $Q_j\gets Q_j \cup \{(b(y_{j, K_j}a), y_{j, K_j}a)\}$ \;
        }
      }
      $K_j \gets K_j + 1$ \;
     }
     \Return  $Q_j$ and $(A, Y_j, K_j, \phi_j)$
  \caption{\textsf{ApplyGenerators}}
  \label{algorithm-apply-gens}
\end{algorithm}

\begin{lem}
  \label{lem-fropin-parallel-0}
  Algorithm~\ref{algorithm-apply-gens} (\textsf{ApplyGenerators}) can be
  performed concurrently on each fragment $(A, Y_j, K_j, \varnothing, \phi_j)$
  of its input. 
\end{lem}
\begin{proof}
  Every value assigned to $\phi_j$ in 
  \textsf{ApplyGenerators} equals a value for $\phi$ defined in
  \textsf{Update}. It is possible that some
  assignments made in \textsf{Update} for $\phi$ cannot be
  made for $\phi_j$ in \textsf{ApplyGenerators}. In particular, in \textsf{Update}
  if $s(y_{j, K_j})a$ is not reduced, then $\phi_j(s(y_{j, K_j}), a)$ is always
  defined in \textsf{Update} but is only defined in some cases
  in \textsf{ApplyGenerators}.  Hence that $\phi_j$ is well-defined
  follows by the proof of Lemma~\ref{lem-fropin-1}.
  
  The values $\phi_i(s(y_{j, K_j}), a)$, $\phi_{m}(f(y_{j, K_j}), p(y))$, and
  $\phi_{n}(w, l(y))$ are read in \textsf{ApplyGenerators} and may belong to
  other fragments.  But $|s(y_{j, K_j})|, |p(y)| < 
  c$ and $\phi_{n}(w, l(y))$ is only used if $n = j$ or $|w| < c$.
  The only value which is written in \textsf{ApplyGenerators}
  is $\phi_j(y_{j, K_j}, a)$, and $|y_{j, K_j}| = c$. 
  It follows that \textsf{ApplyGenerators} only reads values of $\phi_{d}(u,
  a)$ or $\phi_{d}(a, u)$  when $d \not= j$ and $|u| < c$, while the algorithm
  only writes to values of  $\phi_j(u, a)$ when $|u| = c$.  Therefore there are
  no concurrent reads and writes in \textsf{ApplyGenerators}.
\end{proof}

Note that after applying \textsf{ApplyGenerators}, the tuple $(A, Y_j, K_j,
\phi_j)$ is no longer a fragment because Definition~\ref{de-fragment}(d) may
not be satisfied.

The next algorithm, Algorithm~\ref{algorithm-process-queues}
(\textsf{ProcessQueues}), performs the synchronisation step alluded to at the
start of this section. We prove the validity of \textsf{ProcessQueues} in
Lemma~\ref{lem-process-queues-is-valid}.

\begin{algorithm}
\SetKwComment{Comment}{}{}
\DontPrintSemicolon
  \KwIn{The output of Algorithm~\ref{algorithm-apply-gens}
    (\textsf{ApplyGenerators}) for all $m\in \{1, \ldots, k\}$ 
    and $j \in \{1, \ldots, k\}$.}
  \KwOut{A fragment for $S$ extending $(A, Y_j, K_j, \phi_j)$,
      containing every reduced word $wa$ such that $(j, wa) \in Q_1\cup \cdots
      \cup Q_k$.}
    \For(\Comment*[f]{loop over $Q_1\cup\cdots \cup Q_k$ in
    short-lex order on $wa$}){$(b(wa), wa) \in Q_1\cup\cdots\cup Q_k$} {
      \If{$b(wa) = j$} {
        \If{$\nu(wa)= \nu(y)$ for some $y\in Y_j$} {
          $\phi_j(w, a) = y$
        } \Else {
          $Y_j\gets Y_j\cup\{wa\}$ \label{apply-gens-6}\;
          $\phi_j(w, a) := wa$ \;
        }
      }
    }
    \Return  $(A, Y_j, K_j, \phi_j)$
  \caption{\textsf{ProcessQueues}}
  \label{algorithm-process-queues}
\end{algorithm}

\begin{lem}\label{lem-process-queues-is-valid}
  Algorithm~\ref{algorithm-process-queues} (\textsf{ProcessQueues}) returns a
  fragment for $S$ satisfying Definition~\ref{de-fragment}(d)(ii) and 
  it can be performed concurrently on each $j \in \{1, \ldots, k\}$.
\end{lem}
\begin{proof}
  If \textsf{ProcessQueues} is run concurrently in $k$ processes, for distinct
  values of $j$, then each process only writes to $\phi_j$ and only reads from
  $Q_1\cup\cdots \cup Q_k$. Hence \textsf{ProcessQueues} can be performed
  concurrently.

  Parts (a) of Definition~\ref{de-fragment}  holds trivially. For part (b), it
  suffices to note that $wa$ added in line~\ref{apply-gens-6}, is a reduced
  word, since we loop over the elements in $Q_1\cup \cdots\cup Q_k$ in short-lex
  order. 
  
  For part (c), we must show that $1\leq K_j \leq |Y_j| + 1$ and that $y_{j,
  K_j - 1}$ is the maximum word in $Y_j$ of length $|y_{j, K_{j} -1}|$.  Since
  $K_j$ is not modified by \textsf{ProcessQueues} the first condition holds.
  \textsf{ApplyGenerators} must have returned to obtain the input for
  \textsf{ProcessQueues}, which implies that $y_{j, K_j - 1}$ is the maximum
  word in $Y_j$ of length $|y_{j, K_{j} -1}|$ before \textsf{ProcessQueues} is
  called. \textsf{ProcessQueues} only adds words to $Y_j$ of length $c + 1 >
  |y_{j, K_{j} -1}|$ and so $y_{j, K_j - 1}$ is the maximum word in $Y_j$, and
  part (c) holds. 
  
  For part (d), suppose that $w \in Y_1\cup \cdots \cup Y_k$ and $a\in A$ are
  such that $\phi_j(w, a)$ is defined in \textsf{ApplyGenerators} or
  \textsf{ProcessQueues}. In either case, $\phi_j(w, a) \in Y$ and since we
  have shown that every element in $Y$ is reduced, $\phi_j(w, a)$ must be also. 

  The input of \textsf{ApplyGenerators} satisfies
  Lemma~\ref{lem-fropin-parallel-0}, in particular part (iv), which states
  that before applying \textsf{ApplyGenerators} 
  $$\dom(\phi_j) = (A\times \set{y\in Y_j}{|y| < c})\cup
  (\set{y\in Y_j}{|y| < c} \times A).$$
  If $a \in A$ and $w\in Y_n$ such that $|w| = c$ and $b(wa) = j$, then either
  \textsf{ApplyGenerators} defines some $\phi_j(w, a)$ or $(j, wa)$ is placed in
  $Q_n$. In the latter case, $\phi_j(w, a)$ is defined in
  \textsf{ProcessQueues}. 
  Hence when  \textsf{ProcessQueues} returns  
  $$\dom(\phi_j) = \big(A\times \{y_{j, 1}, \ldots, y_{j, L}\}\big)\cup
  \big(\{y_{j,1}, \ldots, y_{j,K_j-1}\} \times A\big)$$ 
  where $L < K_j - 1$ and
  $L$ is the largest value such that $|y_{L}| < |y_{K_j - 1}|$.
  That $\nu(\phi(u, v)) = \nu(uv)$ for all $u, v\in \dom(\phi)$ follows by
  the argument in the proof of Lemma~\ref{lem-fropin-1}.
\end{proof}

We have all of the ingredients to state the concurrent version of Froidure-Pin,
the validity of which is proven in Lemma~\ref{lem-final}.

\begin{algorithm}
\SetKwComment{Comment}{}{}
\DontPrintSemicolon
  \KwIn{A collection of fragments $(A, Y_1, K_1, \phi_1), \ldots, (A, Y_k, K_k,
      \phi_k)$ for a semigroup $S$ satisfying the hypothesis of
      Lemma~\ref{lem-union-fragments} and $M\in \N$.}
  \KwOut{A collection of fragments $(A, Y_1, K_1, \phi_1), \ldots, (A, Y_k, K_k,
      \phi_k)$ for a semigroup $S$ satisfying the hypothesis of
      Lemma~\ref{lem-union-fragments} and with size at least $\min\{M, |S|\}$.}
  \caption{\textsf{ConcurrentFroidurePin}}
  \label{algorithm-parallel}
    $c := \max \set{|y_{i, K_i - 1}|}{1\leq i\leq k,\ K_i > 1}\cup \{1\}$ \;
  \While{$\exists j$ $K_j \leq |Y_j|$ and $|Y_1\cup \ldots \cup Y_k| <
  M$ \label{algorithm-parallel-5} } {
     $Q_1 := \varnothing, \ldots, Q_k = \varnothing$ \;
    \For(\Comment*[f]{This can be done concurrently}){$j\in \{1, \ldots, k\}$} {
        $Q_j, (A, Y_j, K_j, \phi_j)\gets$ \textsf{ApplyGenerators}$(A,
        Y_j, K_j, \phi_j)$\;
    \label{algorithm-parallel-12}}
    \For(\Comment*[f]{This can be done concurrently}){$j\in \{1, \ldots, k\}$ 
        \label{algorithm-parallel-13}} {
        $(A, Y_j, K_j, \phi_j) \gets$\textsf{ProcessQueues}$(Q, j)$\;
    }
    \For(\Comment*[f]{This can be done concurrently}) {$j\in \{1, \ldots, k\}$ 
    \label{closure-if-left}} {
      $L_j = \max\set{i \in \N}{|y_{j, i}| < c,\ y_{j, i} \in Y_j}$\;
      \For(\Comment*[f]{if $\set{y\in Y_j}{|y| = c} = \varnothing$, then $K_j - 1 <
      L_j + 1$})
      {$i\in \{L_j + 1, \ldots, K_j - 1\}$ \label{algorithm-parallel-15}} {
        \For{$a\in A$} {
          $\phi_j(a, y_{j, i}) := \phi_{n}(\phi_{m}(a, p(y_{j,i})),
          l(y_{j,i}))$ where $p(y_{j,i}) \in Y_{m}$ and $\phi_{m}(a,
          p(y_{j,i})) \in Y_{n}$\;
        }
      }
     \label{algorithm-parallel-20}}
  \label{algorithm-parallel-21}
    $c \gets c + 1$ \Comment*[f]{$c\gets \max\set{|y_{i, K_i - 1}|}{1\leq i\leq k,\
    K_i > 1}$}\; 
    }
  \Return 
  $(A, Y_1, K_1, \phi_1), \ldots, (A, Y_k, K_k, \phi_k)$
\end{algorithm}

\begin{lem}\label{lem-final}
  If $(A, Y, K, \varnothing, \phi)$ is a snapshot for a semigroup $S$ and $M\in
  \N$, then Algorithm~\ref{algorithm-parallel} (\textsf{ConcurrentFroidurePin})
  returns a snapshot of $S$ that extends $(A, Y, K, \varnothing, \phi)$ has at
  least $\min\{M, |S|\}$ elements.
\end{lem}
\begin{proof}
  By Lemma~\ref{lem-fropin-parallel-0} and~\ref{lem-process-queues-is-valid},
  when line~\ref{closure-if-left} is reached, every tuple $(A, Y_j, K_j, \phi_j)$
  is a fragment for $S$, and the loops applying \textsf{ApplyGenerators} and
  \textsf{ProcessQueues} can be performed concurrently.

  We will show that by the end of the for-loop started in
  line~\ref{closure-if-left},   $(A, Y_1, K_1, \phi_1), \ldots, (A, Y_k, K_k,
  \phi_k)$  is a collection of fragments satisfying the conditions of
  Lemma~\ref{lem-union-fragments}, and that the steps within the for-loop can
  be executed concurrently for each fragment. That the values assigned to
  $\phi_j$ are valid follows by the argument in the proof of
  Lemma~\ref{lem-fropin-2}.  Suppose that $j\in \{1, \ldots, k\}$ is given.
  When the for-loop started in line~\ref{closure-if-left} is complete, $(A,
  Y_j, K_j, \phi_j)$ clearly satisfies Definition~\ref{de-fragment}(a) to (c),
  since $\phi_j$ is the only component which is modified inside the for-loop.
  Hence it suffices to verify Definition~\ref{de-fragment}(d). That
  $\nu(\phi_j(u, v)) = \nu(uv)$ for all $u, v\in \dom(\phi_j)$ follows again by
  the same argument as in Lemma~\ref{lem-fropin-2}. 
  
  By Lemma~\ref{lem-process-queues-is-valid}, before $\phi_j$ is modified in
  this loop it satisfies Definition~\ref{de-fragment}(d)(ii), i.e.\
  $$\dom(\phi_j) = \big(A\times \{y_{j, 1}, \ldots, y_{j, L}\}\big)\cup
  \big(\{y_{j,1}, \ldots, y_{j,K_j-1}\} \times A\big)$$ where $L < K_j - 1$ is
  the largest value such that $|y_{L}| < |y_{K_j - 1}|$.  After the for-loop
  starting in line~\ref{algorithm-parallel-15}, 
  \begin{equation}\label{eq-1}
    \dom(\phi_j) = \big(A\times \{y_{j, 1}, \ldots, y_{j, K_j-1}\}\big)\cup
  \big(\{y_{j,1}, \ldots, y_{j,K_j-1}\} \times A\big).
  \end{equation}
  In other words, Definition~\ref{de-fragment}(d)(i) holds, and so  $(A, Y_j,
  K_j, \phi_j)$ is a fragment. 
  
  It is clear that the collection of fragments contains at least $\min\{M,
  |S|\}$ elements. Next, we show that the for-loop started in
  line~\ref{closure-if-left} can be executed concurrently. Since $|p(y_{j, i})|
  < |y_{j, i}| = c$ and so $\phi_m(a, p(y_{j, i}))$ is defined before
  line~\ref{closure-if-left}. Furthermore, $|\phi_m(a, p(y_{j, i}))| \leq
  |y_{j, i}| = c$ and by, \eqref{eq-1}, $\phi_n(w, b)$ is defined for all
  reduced $w$ of length at most $c$ and for all $b\in A$. In other words,
  $\phi_n(\phi_m(a, p(y_{j, i})), l(y_{j,i}))$ is defined before
  line~\ref{closure-if-left}.
 
  It remains to show that the collection of fragments satisfies the conditions
  in Lemma~\ref{lem-union-fragments}.
  For Lemma~\ref{lem-union-fragments}(i), 
  a reduced word $y$ belongs to $Y_j$ if and only if $b(y) = j$, and hence
  $Y_i\cap Y_j = \varnothing$ if $i \not= j$.
  To show that  Lemma~\ref{lem-union-fragments}(ii) holds, it suffices to
  show that $\bigcup_{i=1}^{k}Y_i = \set{w\in A^+}{|w| \leq |y_{N}|,\ w\ \text{is
  reduced}}$. Suppose that $w\in A ^+$ is reduced and $|w| \leq |y_{N}|$. Then
  $|p(w)| < |y_{N}|$ and so $p(w) \in \{y_{j, 1}, \ldots, y_{j, K_j - 1}\}$ for
  some $j$. Hence $(p(w), l(w)) \in \dom(\phi_j)$ and so $w = \phi_j(p(w),
  l(w)) \in \bigcup_{i=1}^{k}Y_i$.
  Within \textsf{ConcurrentFroidurePin} the values of $K_j$ are only modified
  in the calls to \textsf{ApplyGenerators}. 
  Every call to \textsf{ApplyGenerators} increases $K_j$ so
  that either $|y_{j, K_{j} - 1}| = c + 1$, or there are no words of
  length $c + 1$ in the $j$th fragment. In other words,
  Lemma~\ref{lem-fropin-parallel-0}(iii) holds.
  We showed in \eqref{eq-1} that Lemma~\ref{lem-union-fragments}(iv) holds
  and Lemma~\ref{lem-union-fragments}(v) holds trivially.  
\end{proof}


\subsection{Experimental results}\label{subsect-parallel-experiment}

In this section we compare the original version of 
Algorithm~\ref{algorithm-froidure-pin} (\textsf{FroidurePin}) as implemented in
\libsemigroups~\cite{Mitchell2016ab}, and the concurrent version in
Algorithm~\ref{algorithm-parallel} (\textsf{ConcurrentFroidurePin}). 
The implementation of \textsf{ConcurrentFroidurePin} in
\libsemigroups~\cite{Mitchell2016ab} uses thread based parallelism using 
C++11 Standard Template Library thread objects. 

We start by comparing the number of products of elements in $S$ that are
actually computed in \textsf{FroidurePin}
and \textsf{ConcurrentFroidurePin}. In~\cite[Theorem 3.2]{Froidure1997aa}, it is
shown that the number of such products in
\textsf{FroidurePin} is $|S| + |R| - |A| - 1$ where $R$ is
the set of relations for $S$ generated by
\textsf{FroidurePin}. One of the main advantages of the
Froidure-Pin Algorithm, concurrent or not, is that it avoids multiplying
elements of $S$ as far as possible by reusing information learned
about $S$ at an earlier stage of the algorithm. This is particularly important
when the complexity of multiplying elements in $S$ is high.
\textsf{ConcurrentFroidurePin} also avoids multiplying elements of $S$, but
is more limited in its reuse of previously obtained information. The number of
products of elements $S$ depends on the number of fragments $k$ used by
\textsf{ConcurrentFroidurePin} and the function $b: A ^ * \to \{1, \ldots,
k\}$.  The \textit{full transformation monoid} $T_n$ of degree $n\in \N$
consists of all functions from $\{1, \ldots, n\}$ to  $\{1, \ldots, n\}$ under
composition of functions. It is generated by the following transformations:
\begin{equation*}
  \begin{pmatrix}
    1 & 2 & 3 & \cdots & n - 1 & n \\
    2 & 3 & 4 & \cdots & n     & 1 
  \end{pmatrix},
  \begin{pmatrix}
    1 & 2 & 3 & \cdots &  n - 1 & n \\
    2 & 1 & 3 & \cdots & n - 1 & n
  \end{pmatrix},
  \begin{pmatrix}
    1 & 2 & 3 & \cdots & n - 1 & n \\
    1 & 2 & 3 & \cdots & n - 1 & 1
  \end{pmatrix}.
\end{equation*}
We compare the number of products of elements of $S$ in
\textsf{FroidurePin} and \textsf{ConcurrentFroidurePin} for each
of  $k = 1, 2, 4, \ldots, 32$ fragments and for the full transformation monoid
of degree $n = 3, \ldots, 8$; see Figure~\ref{fig-table-products}. 
The number of products in \textsf{FroidurePin} is a lower
bound for the number in \textsf{ConcurrentFroidurePin}, and we would not
expect \textsf{ConcurrentFroidurePin} to achieve this bound. However, from
the table in Figure~\ref{fig-table-products} it can be observed that the number
of products in \textsf{ConcurrentFroidurePin} is of the same order of
magnitude as that in \textsf{ConcurrentFroidurePin}.
In~\cite{Froidure1997aa}, it was noted that there are 678223072849 entries in
the multiplication table for $T_7$ but only slightly less than a million
products are required in \textsf{FroidurePin}; we note that
only slightly more than a million products are required in
\textsf{ConcurrentFroidurePin}.

\begin{table}
  \centering
  \begin{tabular}{l||c|c|c|c|c|c}
    $n$ & 3 & 4 & 5 & 6 & 7 & 8 \\ \hline\hline

    $|T_n| = n ^ n$ &  27& 256& 3125& 46656& 823543& 16777216
    \\  \hline
    \textsf{FroidurePin} & 40 & 340 & 3877 & 54592
    & 926136 & 18285899 \\ \hline
    \textsf{ConcurrentFroidurePin} (1 fragments)& 
    45 &405 &  4535 & 66293 & 1048758 &20235231 \\\hline
    \textsf{ConcurrentFroidurePin} (2 fragments)&45 &415& 4586
    & 67835 &  1106562 & 22763829\\ \hline
    \textsf{ConcurrentFroidurePin} (4 fragments)&47 &406&4587&   67682&
    1153668 & 23093948 \\\hline
    \textsf{ConcurrentFroidurePin} (8 fragments)&46 &405&4589&   67433&
    1155484 & 23411798 \\\hline
    \textsf{ConcurrentFroidurePin} (16 fragments)&46 &402&4596&  67578&
    1153832 & 23616000 \\\hline
    \textsf{ConcurrentFroidurePin} (32 fragments)&46 &404&4563&  67755&
    1152818 & 23566915
  \end{tabular}
  \caption{Comparison of the number of products of elements in
  Algorithms~\ref{algorithm-froidure-pin} (\textsf{FroidurePin})
  and~\ref{algorithm-parallel} (\textsf{ConcurrentFroidurePin}) applied to the
  full transformation monoid $T_n$.}
  \label{fig-table-products}
\end{table}

In Figures~\ref{fig-small} and~\ref{fig-big} we plot the performance of
\textsf{ConcurrentFroidurePin} against the number of fragments it uses for
a variety of examples of semigroups $S$. As would be expected, if the semigroup
$S$ is relatively small, then there is no advantage to using
\textsf{ConcurrentFroidurePin}; see Figure~\ref{fig-small}. However, if the
semigroup $S$ is relatively large, then we see an improvement in the runtime of 
\textsf{ConcurrentFroidurePin} against
\textsf{FroidurePin}; see Figure~\ref{fig-big}
and~\ref{fig-reflex}. Note that the monoid of reflexive $5\times 5$ Boolean
matrices has $1414$ generators.

All of the computations in this section were run on a Intel Xeon CPU E5-2640 v4
2.40GHz, 20 physical cores, and 128GB of DDR4 memory.

\begin{figure}[h]
  \centering
  \includegraphics[width=0.9\textwidth]{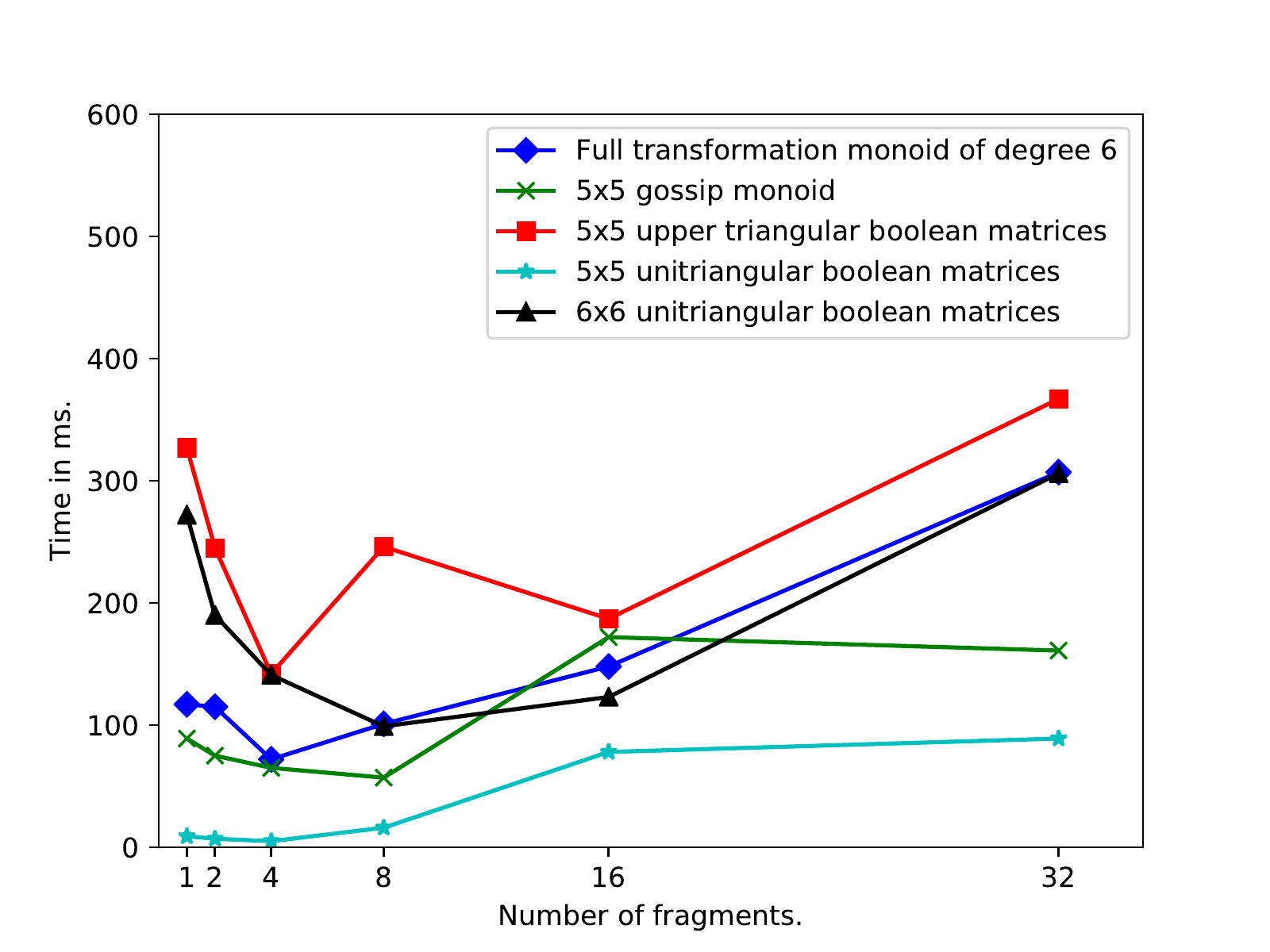}
  \caption{Run times for Algorithm~\ref{algorithm-parallel}
  (\textsf{ConcurrentFroidurePin}) against number of fragments.}
  \label{fig-small}
\end{figure}

\begin{figure}[h]
  \centering
  \includegraphics[width=0.9\textwidth]{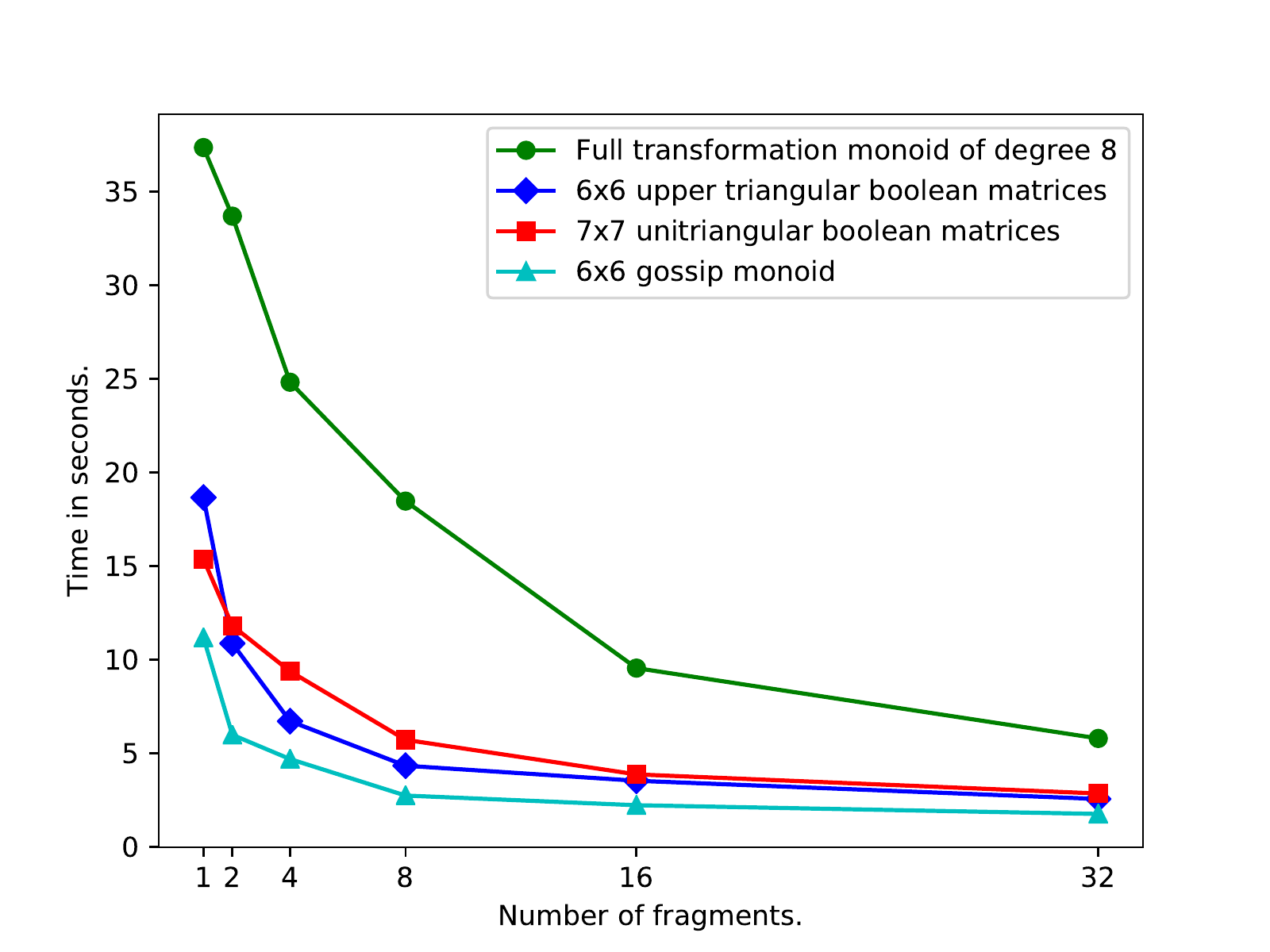}
  \caption{Run times for Algorithm~\ref{algorithm-parallel}
  (\textsf{ConcurrentFroidurePin}) against number of fragments.}
  \label{fig-big}
\end{figure}

\begin{figure}[h]
  \centering
  \includegraphics[width=0.9\textwidth]{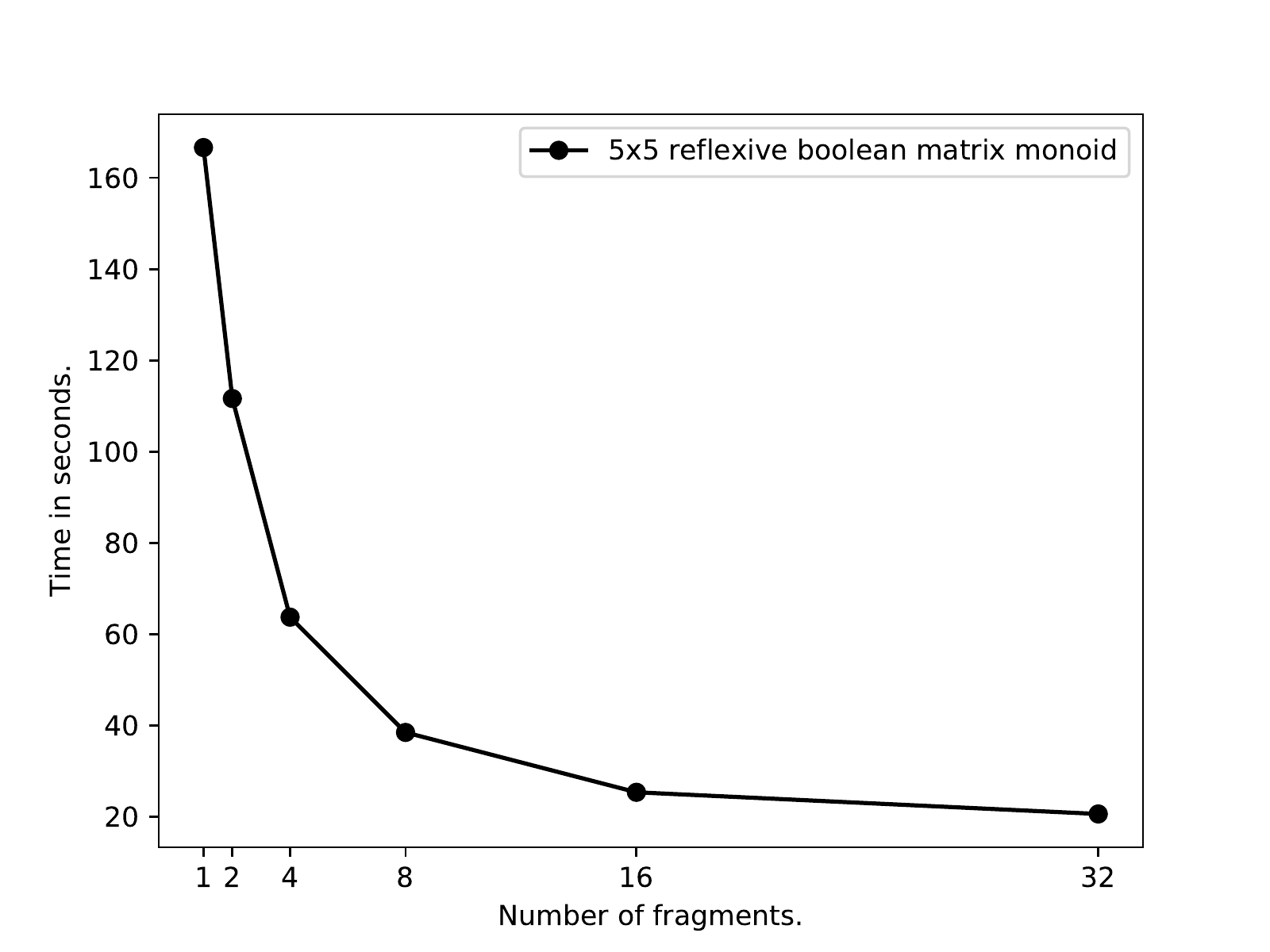}
  \caption{Run times for Algorithm~\ref{algorithm-parallel} against number of
  fragments.}
  \label{fig-reflex}
\end{figure}

\begin{thebibliography}{10}

\bibitem{Donoven2016aa}
C.~Donoven, J.~D. Mitchell, and W.~Wilson,
\newblock Computing maximal subsemigroups of a finite semigroup, submitted
\newblock \url{https://arxiv.org/abs/1606.05583}

\bibitem{East2015ab}
J.~East, A.~Egri-Nagy, J.~D. Mitchell, and Y.~P{\'e}resse,
\newblock Computing finite semigroups,
\newblock to appear in {\em Journal of Symbolic Computation}, 
\newblock \url{http://arxiv.org/abs/1510.01868}

\bibitem{Elliot2017aa}
    Luke~Elliot, Alex~Levine, James~D.~Mitchell, and Nicolas~M.~Thi\'ery,
    \newblock libsemigroups python bindings, 
    \newblock
    \url{https://github.com/james-d-mitchell/libsemigroups-python-bindings}

\bibitem{Froidure1997aa}
V{\'e}ronique Froidure and Jean-Eric Pin,
\newblock Algorithms for computing finite semigroups,
\newblock In {\em Foundations of computational mathematics ({R}io de {J}aneiro,
  1997)}, pages 112--126. Springer, Berlin, 1997.

\bibitem{GAP4}
The GAP~Group,
\newblock {\em {GAP -- Groups, Algorithms, and Programming, Version 4.8.7}}; 2017.  
\newblock \url{http://www.gap-system.org}

\bibitem{B.-Eick2004aa}
D.~Holt with B.~Eick and E.~O'Brien,
\newblock {\em Handbook of computational group theory},
\newblock CRC Press, Boca Raton, Ann Arbor, London, Tokyo, 2004.

\bibitem{Howie1995aa}
John~M. Howie,
\newblock {\em Fundamentals of semigroup theory}, volume~12 of {\em Londo,
  Mathematical Society Monographs. New Series},
\newblock The Clarendon Press Oxford University Press, New York, 1995,
\newblock Oxford Science Publications.

\bibitem{Konieczny1994aa}
Janusz Konieczny,
\newblock Green's equivalences in finite semigroups of binary relations,
\newblock {\em Semigroup Forum}, 48(2):235--252, 1994.

\bibitem{Lallement1990aa}
Gerard Lallement and Robert McFadden,
\newblock On the determination of {G}reen's relations in finite transformation
  semigroups,
\newblock {\em J. Symbolic Comput.}, 10(5):481--498, 1990.

\bibitem{Linton1998aa}
S.~A. Linton, G.~Pfeiffer, E.~F. Robertson, and N.~Ru{\v{s}}kuc,
\newblock Groups and actions in transformation semigroups,
\newblock {\em Math. Z.}, 228(3):435--450, 1998.

\bibitem{Mitchell2016ab}
J.~D. Mitchell et~al,
\newblock {\em libsemigroups - C++ library - version 0.3.1}, May 2017,
\newblock
    \url{https://james-d-mitchell.github.io/libsemigroups/}

\bibitem{Mitchell2016aa}
J.~D. Mitchell et~al,
\newblock {\em Semigroups - GAP package, Version 3.0.1}, June 2017,
    \newblock \url{http://gap-packages.github.io/Semigroups/}

\bibitem{Mitchell2017ac}
  J. Jonu\v{s}as and J.~D. Mitchell, 
  \newblock \textit{Benchmarking libsemigroups}, 
  \newblock \url{https://james-d-mitchell.github.io/2017-06-09-benchmarks/}

\bibitem{Pin2009ab}
Jean-Eric Pin,
\newblock Semigroupe 2.01: a software for computing finite semigroups, April
  2009, \url{https://www.irif.fr/~jep/Logiciels/Semigroupe2.0/semigroupe2.html}

\bibitem{Seress2003ab}
{{\'A}}kos Seress,
\newblock {\em Permutation group algorithms}, volume 152 of {\em Cambridge
  Tracts in Mathematics},
\newblock Cambridge University Press, Cambridge, 2003.

\bibitem{Sims1970aa}
Charles~C. Sims,
\newblock Computational methods in the study of permutation groups,
\newblock In {\em Computational {P}roblems in {A}bstract {A}lgebra ({P}roc.
  {C}onf., {O}xford, 1967)}, pages 169--183. Pergamon, Oxford, 1970.

\bibitem{Sims1994aa}
Charles~C. Sims,
\newblock {\em Computation with finitely presented groups},
\newblock Encyclopedia of mathematics and its applications, Cambridge
  University Press, Cambridge, England, New York, 1994.

\bibitem{Rhodes2009aa}
  J.~Rhodes and B.~Steinberg,
  \newblock \textit{The $q$-theory of finite semigroups},
  \newblock Springer Monographs in Mathematics, New York, 2009.

\end{thebibliography}
\end{document}